\documentclass[12pt, reqno]{amsart}
\usepackage{amssymb}
\usepackage{amsmath}
\usepackage{graphicx, color}
\usepackage{wrapfig,framed}

\usepackage[height=22cm, width=16.5cm, hmarginratio={1:1}]{geometry}
\usepackage{hyperref}

\theoremstyle{plain}
\newtheorem{theorem}{Theorem}
\newtheorem{prop}{Proposition}
\newtheorem{lemma}{Lemma}
\newtheorem{cor}{Corollary}
\theoremstyle{definition}

\newcommand{\beq}{\begin{equation}}
\newcommand{\eeq}{\end{equation}}

\newcommand{\nn}{\nonumber}
\newcommand{\CC}{\mathbb{C}}
\newcommand{\e}{\epsilon}
\newcommand{\p}{\partial}
\newcommand{\F}{\mathcal{F}}
\newcommand{\cH}{\mathcal{H}}
\newcommand{\bbT}{{\bf T}}

\def\={\; = \;}
\def\+{\; + \;}
\def\:={\; := \; }

\begin{document}

\title[On the Hodge-BGW correspondence]{On the Hodge-BGW correspondence}
\author{Di Yang, Qingsheng Zhang}
\begin{abstract}
We establish an explicit relationship between the partition function of certain special cubic Hodge integrals 
and the generalized Br\'ezin--Gross--Witten (BGW) partition function, 
which we refer to as the {\it Hodge-BGW correspondence}. As an application, we obtain 
an ELSV-like formula for generalized BGW correlators.
\end{abstract}

\maketitle

\section{Introduction and statements of the results}
In this paper, we establish a relationship between two celebrated partition functions:
the partition function of certain special cubic Hodge integrals and    
the generalized BGW partition function, that we will call the {\it Hodge-BGW correspondence}.

Let~$\overline{\mathcal{M}}_{g,n}$ denote the Deligne--Mumford moduli 
space of stable algebraic curves of genus~$g$ with~$n$ distinct marked points~\cite{DM}. 
Let $\mathcal L_p$ be the $p$th tautological line bundle on~$\overline{\mathcal{M}}_{g,n}$,  and $\mathbb{E}_{g,n}$ the Hodge bundle.
Denote by $\psi_p:=c_1(\mathcal L_p)$, $p=1,\dots,n$, the first Chern class of $\mathcal L_p$, 
and by  $\lambda_j:= c_j(\mathbb{E}_{g,n})$, $j=0, \dots, g$, the $j$th Chern class of $\mathbb E_{g,n}$. 
We also denote by $\kappa_d:=f_*(\psi_{n+1}^{d+1})$ the $\kappa$-class~\cite{AC,KK,Mum}, 
where $f: \overline{\mathcal{M}}_{g,n+1}\to \overline{\mathcal{M}}_{g,n}$ is the {\it forgetful map}. 
Recall that the {\it intersection numbers of mixed $\psi$-, $\lambda$-, $\kappa$-classes} are integrals on~$\overline{\mathcal{M}}_{g,n}$ of the form
\beq\label{hodgekappaint}
\int_{\overline{\mathcal{M}}_{g,n}} \, \psi_1^{i_1}\cdots \psi_n^{i_n} \, \lambda_1^{j_1} \cdots \lambda_g^{j_g} \,\kappa_{d_1}\cdots\kappa_{d_\ell}\,,
\eeq
where $ i_1, \dots, i_n, j_1, \dots, j_g \geq 0$, $d_1,\cdots,d_\ell \geq 1$, $\ell \geq 0$.
These integrals vanish unless
\beq\label{ddhodge}
(i_1\+\cdots\+ i_n) \+ (j_1\+2 \, j_2\+\cdots\+g \, j_g)\+(d_1\+\cdots\+d_\ell)\= 3 \, g \,-\, 3 \+ n \,. 
\eeq
When $\ell=0$, they are also called {\it Hodge integrals}.

We will be particularly interested in the cubic Hodge integrals of the following form
\beq\label{chbcy}
\int_{\overline{\mathcal{M}}_{g,n}} \, 
\Lambda(-1)^2 \, \Lambda\Bigl(\frac{1}{2}\Bigr) \, \psi_1^{i_1} \cdots \psi_n^{i_n} \,, 
\eeq
where $\Lambda(z) := \sum_{j=0}^g \lambda_j  z^j$ denotes the Chern polynomial of~$\mathbb E_{g,n}$.
These Hodge integrals belong to the class of cubic 
Hodge integrals satisfying the Calabi--Yau condition \cite{GV, LLLZ, MV, OP}.
They also appeared in the Gopakumar--Mari\~no--Vafa conjecture~\cite{GV, MV} regarding 
the Chern--Simons/string duality and in the theory of topological vertex~\cite{LLLZ, LLZ1, LLZ2}.
Recently, they play important roles in the Hodge-GUE correspondence~\cite{DLYZ1,DLYZ2,DY1}, 
which implies the ELSV-like (cf.~\cite{ELSV}) formula for even GUE and modified GUE correlators~\cite{BGF, DY3, GF}; 
see~\cite{A21-1, A21-3, Zhou0} for interesting connections to the KP hierarchy and 2D Toda lattice.
Denote by 
\beq\label{cubichodgeint}
\cH({\bf t};\e) \:= \sum_{g,n\geq0} \,\e^{2g-2}\, \sum_{i_1,\dots,i_n\geq0}\, \frac{t_{i_1}\cdots t_{i_n}}{n!}\, \int_{\overline{\mathcal{M}}_{g,n}}  \Lambda(-1)^2 \, \Lambda\Bigl(\frac{1}{2}\Bigr)\,\psi_1^{i_1}\cdots \psi_n^{i_n}
\eeq
the generating series of the cubic Hodge integrals~\eqref{chbcy}, called 
the {\it Hodge free energy associated to $\Lambda(-1)^2\Lambda(1/2)$}.
Denote by 
\beq\label{hodgepart}
Z_{\rm H}({\bf t};\e) \:= e^{\cH({\bf t};\e)}
\eeq
the partition function, and by~$\cH_g({\bf t})$ the genus~$g$ part of the Hodge free energy, i.e.,
\beq
\cH({\bf t};\e) \= \sum_{g\geq0} \, \e^{2g-2} \, \cH_g({\bf t})\,.
\eeq

We continue and recall terminologies in the BGW side. 
The BGW partition function was introduced in~\cite{BG,GW}.
In~\cite{MMS}, a {\it one parameter deformation} of this partition function, 
called the {\it generalized BGW partition function with the parameter~$N$}, 
was given as a particular generalized Kontsevich model~\cite{KMMM} (cf.~also~\cite{A,BR} for the definition). 
Let us denote this partition function by $Z_{\rm gBGW}(N, {\bf T}; \hbar)$, where 
 ${\bf T}=( T_1, T_3, T_5,\dots)$ is an infinite vector of indeterminates and~$N$ is an indeterminate. 
It is 
often normalized by the following initial condition:
\beq\label{Zgbgw1}
Z_{\rm gBGW}(N, {\bf 0}; \hbar) \;\equiv\; 1 \,. 
\eeq
The logarithm $\log Z_{\rm gBGW}(N, {\bf T}; \hbar) =: \F_{\rm gBGW}(N, {\bf T}; \hbar)$, belonging to $\CC[\hbar][[N,{\bf T}]]$,  
is called the {\it generalized BGW free energy with the parameter~$N$}.
It is known e.g.~in~\cite{A} that the power series $\F_{\rm gBGW}(N, {\bf T}; \hbar)$ has the form
\beq\label{topoexp1214}
\F_{\rm gBGW}(N, {\bf T}; \hbar) \= \sum_{l\ge1} \, \sum_{a_1,\dots,a_l\geq0} \, 
\hbar^{2|a|}\,\sum_{g=0}^{|a|+1} \, (-2)^{|a|-g+1} \, c_g(a_1,\cdots, a_l) \, N^{2|a|-2g+2} \, \frac{\prod_{p=1}^l T_{2a_p+1}}{l!} \,,
\eeq
where $c_{g}(a_1,\cdots, a_l)$ ($g\ge0$, $l\geq1$, $a_1,\dots,a_l\geq0$) are numbers, and $|a|:=\sum_{i=1}^l a_i$.
We call $c_{g}(a_1,\cdots, a_l)$ the {\it generalized BGW correlators of genus~$g$}. 
When $|a|=g-1$, the numbers $c_{g}(a_1,\cdots, a_l)$ are the celebrated {\it BGW correlators of genus~$g$}, as they are Taylor coefficients of the BGW free energy, that is, of $\F_{\rm gBGW}(0, {\bf T}; \hbar) $. The numbers $c_g(a_1,\cdots,a_l)$ vanish when $|a|<g-1$.

According to Alexandrov~\cite{A} (cf.~also~\cite{BR, DN, GN, MMS, YZ}), the power series $Z_{\rm gBGW}(N, {\bf T}; \hbar)$ 
satisfies the following Virasoro constraints:
\beq\label{eqn:vira-gBGW}
L_m^{\rm gBGW} \bigl(Z_{\rm gBGW}(N, {\bf T}; \hbar)\bigr) \= 0 \,, \quad m\geq0 \,.
\eeq
Here $L_m^{\rm gBGW}$, $m\geq0$, are operators given by 
\begin{align}
L_m^{\rm gBGW} \= &
\frac{1}{2^{2m+1}} \, \sum_{a\geq0} \, \frac{(2a+2m+1)!!}{(2a-1)!!} \, \widetilde{T}_{2a+1} \, \frac{\p}{\p T_{2a+2m+1}} \,\nn \\
&\+ \frac{\hbar^2}{2^{2m+2}} \, \sum_{a+b=m-1} \, (2a+1)!! \, (2b+1)!! \, \frac{\p^2}{\p T_{2a+1}\p T_{2b+1}} 
\+ \left(\frac{1}{16} - \frac{N^2}{4}\right) \, \delta_{m,0} \label{Lmgbgwdef}
\end{align}
with 
\beq
\widetilde{T}_{2a+1}  \=   T_{2a+1} \,-\, \delta_{a,0} \,,\quad a\geq0\,.
\eeq
Here we note that the normalizations of the independent variables $T_{2a+1}$ differ by simple factors from those of~\cite{A}.
The operators $L_m^{\rm gBGW}$ satisfy the Virasoro commutation relations: 
\beq\label{viralmbgw1216}
\bigl[L_m^{\rm gBGW}, L_n^{\rm gBGW} \bigr] \= (m-n) \, L_{m+n}^{\rm gBGW} \,. 
\eeq
As we shall see from the uniqueness of solution given in Section~\ref{section2.2}, equations \eqref{Zgbgw1}, \eqref{eqn:vira-gBGW} can be used 
as the defining equations for the generalized BGW partition function with the parameter~$N$, i.e. for $Z_{\rm gBGW}(N, {\bf T}; \hbar)$.

Following Alexandrov~\cite{A}, introduce 
\beq\label{sashaconst}
 x \:= N\, \hbar \,\sqrt{-2}\,.
\eeq
Let us define 
\beq \label{def:gbgw-part}
Z(x, \bbT; \hbar) \:= e^{B(x,\hbar)} \, Z_{\rm gBGW}\Bigl( \, \frac{x}{\hbar \,\sqrt{-2}}, \bbT; \hbar \Bigr)\,,
\eeq
where 
\beq \label{eqn:bernoulli}
B(x, \hbar) \= \frac1{\hbar^2} \biggl(\frac{x^2}{4} \, \log \Bigl(-\frac{x}{2}\Bigr) - \frac38 \, x^2\biggr) \+ \frac1{12} \, \log\Big(-\frac{x}{2}\Big)
\+ \sum_{g\geq 2} \frac{\hbar^{2g-2}}{x^{2g-2}}\frac{(-1)^g \, 2^{g-1} \, B_{2g}}{2g \, (2g-2)}
\eeq
with $B_k$ denoting the $k$th Bernoulli number. 
We call $Z(x, {\bf T}; \hbar)$ the {\it generalized BGW partition function}, and call its logarithm $\log Z(x, {\bf T}; \hbar)$  
the {\it generalized BGW free energy}, denoted by $\F(x, {\bf T};\hbar)$. 
Explicitly, by~\eqref{topoexp1214} we have
\begin{align}
\F(x,\bbT;\hbar) \= & \frac1{\hbar^2} \biggl(\frac{x^2}{4} \, \log \Bigl(-\frac{x}{2}\Bigr) - \frac38 \, x^2\biggr) \+ \frac1{12} \, \log\Big(-\frac{x}{2}\Big)
\+ \sum_{g\geq 2} \frac{\hbar^{2g-2}}{x^{2g-2}}\frac{(-1)^g \, 2^{g-1} \, B_{2g}}{2g \, (2g-2)} \nn\\
& \+ \sum_{g\geq0} \, \hbar^{2g-2} \, \sum_{l\geq1} \, \sum_{a_1,\dots,a_l\geq0 \atop |a|\geq g-1} \, c_g(a_1,\dots,a_l) \, x^{2|a|-2g+2} \, \frac{\prod_{p=1}^l T_{2a_p+1}}{l!} \,. \label{genusexpansion}
\end{align}

Inspired by the Hodge-GUE correspondence~\cite{DLYZ2, DY1}, we establish in the following theorem an explicit 
relationship between the cubic Hodge partition function associated to $\Lambda(-1)^2\Lambda(1/2)$ 
and the generalized BGW partition function. 
\begin{theorem}[Main Theorem] \label{main}
The following identity 
\begin{align}\label{mainid}
e^{\frac{A( x,\bbT)}{\hbar^2}} Z_{\rm H}\bigl({\bf t}( x, \bbT); \hbar\,\sqrt{-4}  \bigr) \=  Z( x, \bbT; \hbar) 
\end{align}
holds true in $\CC((\hbar^2))[[x+2]][[{\bf T}]]$.
Here,
\begin{align}\label{deftT}
&  t_i( x, \bbT) \= \delta_{i,0} \, x \+\delta_{i,1}\, -\, \Big(-\frac{1}{2}\Big)^{i-1}\,-\, 2\, \sum_{a\geq0} \,\Big(-\frac{2a+1}{2}\Big)^i  \;  \frac{ {T}_{2a+1} }{a!} \, , \quad i\geq 0\,,
\end{align}
and $A(x, \bbT)$ is a quadratic series given by
\beq\label{defA1215}
A(x,\bbT) \= \frac{1}{2} \, \sum_{a,b\geq0}
\, \frac{\widetilde { T}_{2a+1} \,\widetilde { T}_{2b+1}}{a! \, b! \, (a+b+1)} 
\, - \, \sum_{b\geq0} \, \frac{ x \, \widetilde { T}_{2b+1}}{b!\, (2b+1)}  \,.
\eeq
\end{theorem}
\noindent The proof of this theorem is in Section~\ref{section3}. 
We call~\eqref{mainid} the {\it Hodge-BGW correspondence}.

Let us now present by the following proposition an application of the Hodge-BGW correspondence, expressing 
an arbitrary generalized BGW correlators of genus $g$ in terms of intersection numbers on the moduli space of curves.  
\begin{prop}\label{thm:gbgw-elsv}
For $g\geq 0$, $l\geq 1$, and for 
non-negative integers $a_1,\cdots, a_l$ such that $|a|\geq g-1$, the generalized BGW correlators of genus~$g$ 
are related to intersection numbers by 
\begin{align}
	&c_{g}(a_1,\cdots,a_l) \, \nn\\
	\= & \frac{(-1)^{g-1+l} \, 2^{2g-2+l}}{(2|a|-2g+2)! \, \prod_{p=1}^la_p!} \, \int_{\overline{\mathcal{M}}_{g,l+2|a|-2g+2}}  \Lambda(-1)^2 \, \Lambda\Bigl(\frac{1}{2}\Bigr) \, e^{\sum_{d\geq 1}\frac{(-1)^{d-1}\, \kappa_{d}}{ 2^{d}\, d }}
	\prod_{p=1}^l \frac{1}{1+\frac{2a_p+1}{2}\psi_p}\, . \label{elsvgen1215}
\end{align}
Moreover, for $g\geq2$, we have
\begin{align}
 \int_{\overline{\mathcal{M}}_{g,0}}  \Lambda(-1)^2 \, \Lambda\Bigl(\frac{1}{2}\Bigr) \, 
e^{\sum_{d\geq 1}\frac{(-1)^{d-1}\, \kappa_{d}}{ 2^{d}\, d }} \= - \, \frac{B_{2g}}{2g \, (2g-2) \, 8^{g-1}} \,. \label{elsvgen21215}
\end{align}
\end{prop}
\noindent The proof is given in Section~\ref{section4}.

The following corollary follows straightforwardly from Proposition~\ref{thm:gbgw-elsv}.
\begin{cor}\label{corbgw}
For $g\geq 0$, $l\geq 1$, and for 
non-negative integers $a_1,\cdots, a_l$ such that $|a|=g-1$, the BGW correlators of genus~$g$ 
admit the ELSV-like formula:
\begin{align}
	c_{g}(a_1,\cdots,a_l) 
	\=  \frac{(-1)^{g-1+l} \, 2^{2g-2+l}}{\prod_{p=1}^la_p!} \, \int_{\overline{\mathcal{M}}_{g,l}} 
	\Lambda(-1)^2 \, \Lambda\Bigl(\frac{1}{2}\Bigr) \, e^{\sum_{d\geq 1}\frac{(-1)^{d-1}\, \kappa_{d}}{ 2^{d}\, d }} \, 
	\prod_{p=1}^l \frac{1}{1+\frac{2a_p+1}{2}\psi_p} \, . \label{eqn:bgw-elsv}
\end{align}
\end{cor}

The rest of the paper is organized as follows. In Section~\ref{section2} we review in more details about  
the Hodge partition function and the generalized BGW partition function. In Section~\ref{section3} we prove Theorem~\ref{main}. 
In Section~\ref{section4} we prove Proposition~\ref{thm:gbgw-elsv}.  
Three more applications of the Hodge-BGW correspondence are given in Appendix~\ref{appa}.

\section{Review on Hodge integrals and generalized BGW correlators}\label{section2}

\subsection{$\psi$-class intersection numbers and cubic Hodge integrals}
When $j_1,\dots,j_g$ and~$\ell$ in~\eqref{hodgekappaint} all equal~0, 
the integrals in~\eqref{hodgekappaint} are also called the {\it $\psi$-class intersection numbers}. It was conjectured by  
Witten~\cite{W}, and proved by Kontsevich~\cite{K} that 
the partition function of the $\psi$-class intersection numbers 
is a particular tau-function for the Korteweg--de Vries (KdV) hierarchy, now 
known as the Witten--Kontsevich theorem. To be precise, let 
$Z_{\rm WK}({\bf t}; \e)$ denote this partition function:
\beq\label{zwkdef}
Z_{\rm WK}({\bf t}; \e) \:= \exp\left(\sum_{g\geq 0} \e^{2g-2} \, \sum_{n\geq 0} \, \frac1{n!} \, \sum_{i_1, \dots, i_n\geq 0} \, 
\int_{\overline{\mathcal{M}}_{g, n}} \, \psi_1^{i_1} \cdots \psi_n^{i_n} \, t_{i_1}\cdots t_{i_n}\right) \,.
\eeq
Then $u=u_{\rm WK}({\bf t}; \e):=\e^2 \p_{t_0}^2 \bigl(\log Z_{\rm WK}({\bf t}; \e)\bigr)$ satisfies the KdV hierarchy:
\beq\label{kdv1214}
\frac{\p u}{\p t_i} \= \frac1{(2i+1)!!} \, \biggl[\Bigl(L^{\frac{2i+1}2}\Bigr)_+,L\biggr] \,, \quad i\geq0\,,
\eeq
where $L:=\e^2 \p_{t_0}^2 + 2 \, u$ is the Lax operator of the KdV hierarchy (cf.~e.g.~\cite{Dickey}). Moreover, $Z_{\rm WK}({\bf t}; \e)$ satisfies 
the dilaton and string equations
\begin{align}
& \sum_{i=0}^\infty t_i \, \frac{\p Z_{\rm WK}({\bf t}; \e)}{\p t_i} \+ \e \, \frac{\p Z_{\rm WK}}{\p \e} \+ \frac1{24} \, Z_{\rm WK}({\bf t}; \e) 
\= \frac{\p Z_{\rm WK}({\bf t}; \e)}{\p t_1} \,,  \label{dilatoneq}\\
& \sum_{i=0}^\infty t_{i+1} \frac{\p Z_{\rm WK}({\bf t}; \e)}{\p t_i} \+ \frac{t_0^2}{2 \, \e^2} \, Z_{\rm WK} \= \frac{\p Z_{\rm WK}({\bf t}; \e)}{\p t_0} \, . \label{stringeq}
\end{align}
(For more about a KdV tau-function cf.~\cite{Dickey, DYZ, DZ-norm}.)

According to Dijkgraaf, Verlinde, Verlinde~\cite{DVV}, the Witten--Kontsevich theorem can be equivalently formulated 
as follows: the power series $Z_{\rm WK}({\bf t}; \epsilon)$ satisfies the infinite family of linear equations
\beq \label{Vira-WK}
L_k^{\rm WK} \bigl(\e^{-1} \,\tilde {\bf t}, \e \, \p/\p {\bf t}\bigr) \, \Bigl(Z_{\rm WK}({\bf t}; \epsilon)\Bigr) \= 0 \,, \quad k\geq -1 \,,
\eeq
where $\tilde t_i:=t_i-\delta_{i,1}$, and $L_k^{\rm WK}=L_k^{\rm WK}\bigl( \e^{-1} \, \tilde {\bf t}, \e \, \p/\p {\bf t}\bigr)$, $k\geq -1$, are given by
\begin{align}
& L^{\rm WK}_{-1} \= \sum_{i\geq 1} \, \tilde t_i \, \frac{\p}{\p t_{i-1}} \+ \frac{t_0^2}{2 \, \epsilon^2} \,, \label{virakdvintro1} \\
& L^{\rm WK}_0 \= \sum_{i\geq 0} \, \frac{2i+1}{2} \, \tilde t_i \, \frac{\p}{\p t_{i}} \+ \frac{1}{16} \,,  \label{virakdvintro2}\\
& L^{\rm WK}_k \= \frac{\epsilon^2}{2} \, \sum_{i+j=k-1} \, \frac{(2i+1)!! \, (2j+1)!!}{2^{k+1}}
\frac{\p^2}{\p t_i \p t_j} \+ \sum_{i\geq 0} \, \frac{(2i+2k+1)!!}{2^{k+1}\, (2i-1)!!} \, \tilde t_i \, \frac{\p}{\p t_{i+k}} \,, 
\quad k\geq 1 \,. \label{virakdvintro3}
\end{align}
These operators $L_k^{\rm WK}$ satisfy the following Virasoro commutation relations:
\beq
\Bigl[L^{\rm WK}_k, L^{\rm WK}_l \Bigr] \= (k-l) \, L^{\rm WK}_{k+l}, \quad \forall \, k,l\geq -1 \, . \label{viracommkdv}
\eeq
Equations~\eqref{Vira-WK} are called the {\it Virasoro constraints} for $Z_{\rm WK}({\bf t}; \epsilon)$. Clearly, the $L_{-1}^{\rm WK}$-constraint 
coincides with the string equation~\eqref{stringeq}.

The Hodge partition function defined by \eqref{cubichodgeint}, \eqref{hodgepart} can be obtained 
from the Witten--Kontsevich partition function via the Faber--Pandharipande formula~\cite{FP}:
\beq\label{fphwk}
Z_{\rm H} ({\bf t}; \e) \= \exp\biggl(\sum_{j = 1}^\infty  \frac{(2^{-2j}-1) \, B_{2j}}{j \, (2j-1)} \, D_j\bigl(\e^{-1} \, \tilde {\bf t}, \e \, \p/\p {\bf t}\bigr)\biggr)\, \bigl(Z_{WK}({\bf t}; \e)\bigr) \,,
\eeq
where  
\beq\label{dj}
D_j\bigl(\e^{-1} \, \tilde {\bf t}, \e \p/\p {\bf t}\bigr) \:= - \, \sum_{i\geq 0} \, \tilde t_i \, \frac{\p }{\p t_{i+2j-1}} 
\+ \frac{\e^2}2 \, \sum_{a=0}^{2j-2} \, (-1)^a \, \frac{\p^2 }{\p t_a \p t_{2j-2-a}} \,,\quad j\geq 1\,.
\eeq
It satisfies the following {\it dilaton equation}:
\beq
\widetilde L^{\rm H}_{\rm dilaton} (Z_{\rm H}({\bf t};\e)) \= 0\,,
\label{dilatoneqhodge}\\
\eeq
where $\widetilde L^{\rm H}_{\rm dilaton}$ is the linear operator defined by 
\beq
\widetilde L^{\rm H}_{\rm dilaton} \:= \sum_{i\geq 0} \,\tilde t_{i} \, \frac{\p}{\p t_{i}} \+ \epsilon\,\frac{\p }{\p \epsilon}\+\frac{1}{24} \,.
\eeq

Following~\cite{Zhou1} define the operators $L_k^{\rm H}\bigl(\e^{-1} \, \tilde {\bf t}, \e \, \p/\p {\bf t}\bigr)$ by 
\begin{equation}
L_k^{\rm H}\bigl(\e^{-1} \, \tilde {\bf t}, \e \, \p/\p {\bf t}\bigr) \= e^{G}  \circ L_k^{\rm WK} \bigl(\e^{-1} \, \tilde {\bf t}, \e \, \p/\p {\bf t}\bigr) \circ e^{- G}\,, \quad k\ge -1\,,\label{deflmcubic}
\end{equation}
where
\beq
G \:= \sum_{j = 1}^\infty \, \frac{(2^{-2j}-1) \, B_{2j}}{j \, (2j-1)}  \, D_j\bigl(\e^{-1} \, \tilde {\bf t}, \e \, \p/\p {\bf t}\bigr)\,.
\eeq
Then we have
\beq
\Bigl[L_k^{\rm H}, L_l^{\rm H}\Bigr] \= (k-l) \, L_{k+l}^{\rm H}\,,\quad \forall\,k,l\geq -1 \,, 
\eeq
and moreover, 
\beq\label{vira-Hodge}
L_k^{\rm H} \bigl( \e^{-1}\,\tilde{\bf t}, \e\, \p/\p{\bf t} \bigr) \, \bigl(Z_{\rm H}({\bf t};\e)\bigr) \= 0\,,\quad \forall\, k\geq -1\,.
\eeq

A powerful tool for manipulating Virasoro type operators was introduced by Givental~\cite{Givental}. Let 
us give a short review as it will be used in Section~\ref{section3}.  Convention of the notations will follow from those of~\cite{DLYZ2}.
Denote by $\mathcal{V}$ the space of Laurent polynomials in~$z$ with coefficients in~$\mathbb{C}$.  
On $\mathcal{V}$ there defines a symplectic bilinear form $\omega$:
$$
\omega(f,g) \:= - \, {\rm Res}_{z=\infty } \, f(-z) \, g(z) \, \frac{dz}{z^2} \= - \, \omega(g,f) \,, \quad \forall \,  f,g\in \mathcal{V} \,. 
$$
The pair $(\mathcal{V},\omega)$ is called the {\it Givental symplectic space}. For any $f\in \mathcal{V}$, write
$$
f \= \sum_{i\geq 0} \, q_i \, z^{-i} \+ \sum_{i\geq 0} \, p_i \, (-z)^{i+1} \,.
$$
Then $\{q_i,\,p_i \mid i\geq0\}$ gives the Darboux coordinates for $(\mathcal{V},\omega)$. 

For any infinitesimal symplectic transformation~$A$ on $(\mathcal{V},\omega)$,
define the {\it Hamiltonian associated to~$A$} by
$$H_{A}(f) \= \frac12 \, \omega(f, A (f) ) \= - \, \frac12 \, {\rm Res}_{z=\infty } \, f(-z) \, A (f(z)) \, \frac{dz}{z^2} \, .$$
This Hamiltonian is a {\it quadratic} function on~$\mathcal{V}$, and its quantization is defined via
$$ \widehat{p_i p_j} \= \e^2 \, \frac{\p^2}{\p q_i \p q_j} ,\quad \widehat{p_i q_j} \= q_j \, \frac{\p}{\p q_i} , \quad 
\widehat{q_i q_j} \= \frac1{\e^2} \, q_i q_j \,. $$
Denote the quantization of $H_{A}$ by~$\widehat{A}$. For two infinitesimal symplectic transformations
$A,B$, we have 
$$\Bigl[\widehat{A} , \widehat{B}\Bigr] \= \widehat{[A, B]} \+ \mathcal{C}(H_A , H_B) \,, $$
where $\mathcal{C}$ is the $2$-cocycle term satisfying
$$ 
\mathcal{C}(p_ip_j, q_k q_l) \= - \, \mathcal{C}(q_k q_l, p_i p_j) \= \delta_{i,k} \, \delta_{j,l} \+ \delta_{i,l} \, \delta_{j,k} \,, 
$$
and $\mathcal{C}=0$ for all other pairs of quadratic monomials of $p,q$.
 
Following Givental~\cite{Givental}, define the operators $\ell_k$ by
\beq
\ell_k \= (-1)^{k+1} \, z^{3/2} \circ \p_z^{k+1} \circ z^{-1/2} \,, \quad k\geq -1 \, .
\eeq
It is shown in~\cite{Givental} that 
$\ell_k$ are infinitesimal symplectic transformations on~$\mathcal{V}$, 
and their quantizations coincide with 
the Virasoro operators~\eqref{virakdvintro1}--\eqref{virakdvintro3}. More precisely, 
\beq\label{kdvviragivental}
L_k^{\rm WK} \bigl( \e^{-1}\, \tilde{\bf t}, \e \,\p/\p{\bf t}\bigr) \= \widehat{\ell_k}\big|_{q_i\mapsto \tilde t_i, \p_{q_i} 
\mapsto \p_{t_i}, i\geq 0} \+ \frac{\delta_{k,0}}{16} \,, \quad k\geq -1 \,.
\eeq
Givental~\cite{Givental} also shows that the quantizations of 
the multiplication operators $z^{1-2j}$, $j\geq 1$, coincide with the operators 
$D_j$, $j\geq 1$, defined in~\eqref{dj}, i.e.,
\beq\label{djgivental}
D_j \= \widehat{z^{1-2j}}\big|_{q_i\mapsto\tilde t_i, \p_{q_i} \mapsto \p_{t_i}, i\geq 0} \,.
\eeq
Let 
\beq
\phi(z) \:= \sum_{k=1}^\infty \, \frac{(2^{-2k}-1) \, B_{2k}}{k \, (2k-1)} \, \frac{1}{z^{2k-1}} \,.
\eeq
Clearly, $\phi(z)$ gives an infinitesimal symplectic transformation on~$\mathcal{V}$. Define $\Phi(z):=e^{\phi(z)}$. 
The quantization $\widehat\Phi$ of the symplectomorphism $f(z)\mapsto \Phi(z) f(z)$ 
is defined by 
\beq\label{defbigPhi}
\widehat{\Phi} \:= e^{\widehat{\phi(z)}}\big|_{q_i\mapsto \tilde t_i, \, \p_{q_i} \mapsto \p_{t_i}, \, i\geq 0} \,.
\eeq

We see from~\eqref{dj}, \eqref{djgivental}, \eqref{defbigPhi} that $\widehat{\Phi}=e^G$. 
So the operators~$L_k^{\rm H}\bigl(\e^{-1}\,\tilde {\bf t}, \e\, \p/\p {\bf t}\bigr)$ defined in~\eqref{deflmcubic} have the expressions 
$$
L_k^{\rm H} \bigl( \e^{-1}\,\tilde {\bf t}, \e\, \p/\p{\bf t}\bigr) \= \widehat\Phi \circ L_k^{\rm WK} \, \bigl( \e^{-1}\,\tilde {\bf t}, \e\, \p/\p{\bf t}\bigr) \circ \widehat{\Phi}^{-1} \,, \quad k\geq -1 \,.
$$
Thus by using~\eqref{kdvviragivental} we obtain
\beq \label{tobesim}
L_k^{\rm H} \bigl(\e^{-1}\,\tilde {\bf t}, \e \,\p/\p{\bf t}\bigr) \= 
\left. \left[\widehat{\Phi}\circ\left(\widehat{\ell_k} 
+ \frac{\delta_{k, 0}}{16}\right)\circ \widehat{\Phi}^{-1} \right]\right|_{q_i\mapsto \tilde t_i, \, \p_{q_i} \mapsto \p_{t_i},\,i\geq 0} , \quad k\geq -1 \,.
\eeq
Therefore, as in~\cite{DLYZ2}, we have
\beq\label{LkHgivental}
L_k^{\rm H} \bigl(\e^{-1}\,\tilde {\bf t}, \e\, \p/\p{\bf t}\bigr) \=  
\widehat{\Phi}_k \big|_{q_i\mapsto \tilde t_i, \, \p_{q_i} \mapsto \p_{t_i},\,i\geq 0} \+ \frac{\delta_{k,0}}{16} \,-\, \frac{ \delta_{k, -1} }{16}\,,\quad k\geq -1 \,,
\eeq
where $\Phi_k(z)=\Phi(z) \circ \ell_k \circ \Phi(z)^{-1}$.

\subsection{The generalized BGW partition function}\label{section2.2}

It is proved in~\cite{BR,MMS} 
that the generalized BGW partition function with the parameter~$N$ (i.e. the power series $Z_{\rm gBGW}(N,\bbT;\hbar)$ from the Introduction) 
is a {\it tau-function} of the KdV hierarchy uniquely specified by~\eqref{Zgbgw1} 
and the $m=0$ case of~\eqref{eqn:vira-gBGW}; see~also~\cite{A, DN, DYZ}.
In particular, the power series $u=u_{\rm gBGW}(N,\bbT;\hbar)=\hbar^2 \p_{T_1}^2 \bigl(\log Z_{\rm gBGW}(N,\bbT;\hbar)\bigr)$ satisfies the following equations:
\beq\label{KdVThbar}
\frac{\p u}{\p T_{2a+1}} \= \frac1{(2a+1)!!} \, \biggl[\Bigl(L^{\frac{2a+1}2}\Bigr)_+,L\biggr] \,, \quad a\geq0\,.
\eeq
Here $L=\hbar^2\p_{T_1}^2+2u$. The initial value of this power series is given as follows~\cite{A,DYZ}:
\beq\label{initialugBGW202218}
u_{\rm gBGW}(N, T_1,T_3=T_5=\cdots =0;\hbar) \=\hbar^2\,\frac{\frac{1}{8}-\frac{N^2}{2}}{(1-T_1)^2}\,.
\eeq
Using these properties, let us give a new proof that $\F_{\rm gBGW}(N,\bbT;\hbar)$ has the form given by~\eqref{topoexp1214}. 
Indeed, when $l\geq 2$, this statement follows from the tau-structure of the KdV hierarchy~\cite{Dickey,DYZ,DZ-norm}. For $l=1$, the statement then follows from the $m=0$ case of~\eqref{eqn:vira-gBGW}.

As we recall from Introduction, with Alexandrov's variable~$x$ (see~\eqref{sashaconst}), one  
recognizes that the expansion~\eqref{genusexpansion} for the generalized BGW free energy $\F(x, {\bf T}; \hbar)$ 
is a genus expansion, and it is convenient to write  
\beq
\F(x, {\bf T}; \hbar) \;=:\; \sum_{g\geq0} \, \hbar^{2g-2} \, \F_{g}(x, {\bf T}) \,.
\eeq
We call $\F_g(x, {\bf T})$ the {\it genus~$g$ part of the generalized BGW free energy}. 
The derivatives of 
$\F(x, {\bf T}; \hbar)$ at ${\bf T}={\bf 0}$
\beq
\left.\frac{\p^l \F(x, {\bf T}; \hbar)}{\p T_{2a_1+1} \dots \p T_{2a_l+1}}\right|_{{\bf T}={\bf 0}} \; =: \; 
\langle\sigma_{2a_1+1}\cdots \sigma_{2a_l+1}\rangle(x,\hbar)
\eeq
are called the {\it generalized BGW correlators}, where $l,a_1,\dots,a_l\geq0$. Clearly, for $l\geq1$,
\beq\label{eqn:gbgw-corr}
\langle\sigma_{2a_1+1}\cdots \sigma_{2a_l+1}\rangle(x,\hbar) \= 
\sum_{g\geq0} \, \hbar^{2g-2} \, c_g(a_1,\cdots, a_l) \, x^{2|a|-2g+2} \,.
\eeq

It is clear from~\eqref{def:gbgw-part} that for the generalized BGW partition function $Z(x,\bbT;\hbar)$ (see~\eqref{def:gbgw-part}) the 
Virasoro constraints~\eqref{eqn:vira-gBGW} translate into
\beq\label{eqn:vira-gBGW-x}
L_m \left(Z(x, {\bf T}; \hbar)\right) \= 0 \,, \quad m\geq0 \,,
\eeq
where the operators $L_m$ are defined by
\begin{align}
L_m \:= &
 \sum_{a\geq0} \frac{(2a+2m+1)!!}{2^{m+1}(2a-1)!!} \, \widetilde {T}_{2a+1} \, \frac{\p}{\p T_{2a+2m+1}} \,\nn \\
 &\+ \frac{\hbar^2}2 \sum_{a+b=m-1}\frac{(2a+1)!! \, (2b+1)!!}{2^{m+1}} \frac{\p^2}{\p T_{2a+1}\p T_{2b+1}} 
\+ \biggl(\frac{1}{16} \+ \frac{x^2}{8 \, \hbar^2}\biggr) \, \delta_{m,0} \,. \label{viraLmxT1214}
\end{align}

We have the following lemma.

\begin{lemma} The generalized BGW partition function $Z(x,{\bf T};\hbar)$ satisfies the dilaton equation:
\beq\label{eqn:dilaton}
	L_{\rm dilaton} (Z(x,{\bf T};\hbar)) \= 0 \,,
\eeq
where $L_{\rm dilaton}$ is the linear operator defined by
\beq
L_{\rm dilaton} \:= \sum_{a\geq 0}\,\widetilde T_{2a+1}\,\frac{\p}{\p T_{2a+1}}\+x\,\frac{\p}{\p x}\+\hbar\,\frac{\p }{\p \hbar}\+\frac{1}{24} \,.
\eeq
\end{lemma}
\begin{proof}
	On one hand, by using~\eqref{eqn:bernoulli} and~\eqref{genusexpansion} we have
	\beq
	-x\,\frac{\p \mathcal F_g(x,{\bf T})}{\p x}\+\sum_{a\geq 0}\, 2a\, \widetilde T_{2a+1}\,\frac{\p \mathcal F_g(x,{\bf T})}{\p T_{2a+1}}\= (2g-2)\, \mathcal F_g(x,{\bf T})\,-\,\frac{x^2}{4}\,\delta_{g,0}\,-\,\frac{1}{12}\,\delta_{g,1}\,,
	\eeq
	which gives
	\beq\label{eqn:deg-z}
	x\,\frac{\p Z}{\p x}\,-\,\sum_{a\geq 0}\, 2a\, \widetilde T_{2a+1}\,\frac{\p Z}{\p T_{2a+1}} \+ \hbar \, \frac{\p Z}{\p \hbar} \,-\, \bigg(\frac{x^2}{4\,\hbar^2}\+\frac{1}{12}\bigg) Z\=0\,.
	\eeq
	On another hand, the $m=0$ case of~\eqref{viraLmxT1214} reads
	\beq\label{eqn:l0-z}
	\sum_{a\geq 0}\,\frac{2a+1}{2}\,\widetilde T_{2a+1}\,\frac{\p Z(x,{\bf T};\hbar)}{\p T_{2a+1}}\+ \biggl(\frac{1}{16} 
	\+ \frac{x^2}{8 \, \hbar^2}\biggr)\, Z(x,{\bf T};\hbar) \= 0 \,.
	\eeq
	The lemma follows from \eqref{eqn:deg-z} and \eqref{eqn:l0-z}.
\end{proof}

Let us now prove the following lemma.
\begin{lemma}[\cite{A, BR}] \label{cajZ}
The solution in $\CC((\hbar^2))[[x+2]][[{\bf T}]]$ 
to the Virasoro constraints 
\beq\label{111vira1214}
L_m \bigl( \widehat Z(x,\bbT;\hbar) \bigr) \= 0
\eeq with the initial value
$\widehat Z (x,{\bf 0}; \hbar) = 1$
is unique. 
\end{lemma}

\begin{proof}
Following the same arguments as those in~\cite{A} we see that if a solution $\widehat Z$ 
in $\CC((\hbar^2))[[x+2, {\bf T}]]$ to equations~\eqref{111vira1214} 
exists, it must have the following cut-and-join representation:
\beq\label{eqn:caj}
\widehat Z(x,{\bf T}; \hbar) \= e^{\widehat W}(1)\, ,
\eeq
where
\begin{align}
\widehat W \= & \sum_{a,b\geq0} \,\frac{(2a+2b+1)!!}{(2a-1)!!(2b-1)!!} \,T_{2a+1} \,T_{2b+1}\, \frac{\p}{\p  T_{2a+2b+1}} \,\nn\\
& \+ \frac{\hbar^2}2\, \sum_{a,b\geq0}\,\frac{(2a+1)!!(2b+1)!!}{(2a+2b+1)!!} \,T_{2a+2b+3} \,\frac{\p^2}{\p T_{2a+1}\p T_{2b+1}} 
\+ \left(\frac{1}{8} \+ \frac{x^2}{4\,\hbar^2}\right) T_1 \,.
\end{align}
The lemma is proved.
\end{proof}
\noindent Lemma~\ref{cajZ} also easily follows from a similar argument used in e.g.~\cite{DVV,LYZ}.

We note that the power series 
\[ Z_{\rm gBGW}\Bigl(\frac{x}{\hbar \, \sqrt{-2}},\bbT; \hbar\Bigr) \,\in\, \CC((\hbar^2))[x][[\bbT]] \, \subset \, \CC((\hbar^2))[[x+2]][[\bbT]]\] 
is a solution to~\eqref{111vira1214} satisfying 
$Z_{\rm gBGW}\bigl(\frac{x}{\hbar \, \sqrt{-2}}, {\bf 0}; \hbar)=1$.

\section{Proof of the Main Theorem}\label{section3}

In this section, we prove the Main Theorem. The proof will be similar to that of~\cite{DLYZ2} for the Hodge-GUE correspondence. 

Before entering into the details of the proof, we do several preparations. 
Let us introduce the following linear combinations of $L_k^{\rm H}$, $k\geq-1$:
\begin{equation}\label{wlmdef}
\widetilde L^{\rm H}_{m} \bigl(\e^{-1} \, \tilde {\bf t}, \e \, \p/\p {\bf t}\bigr) \:= - \, \sum_{k=-1}^\infty \, 
\frac{(-m)^{k+1}}{(k+1)!} \, L_k^{\rm H}\bigl(\e^{-1} \, \tilde {\bf t}, \e \, \p/\p {\bf t}\bigr) \,,\quad m\geq 0 \,. 
\end{equation} 

\begin{lemma} \label{alg_lem}
	For a basis $\{\alpha_k \, | \, k\geq -1\}$ of an infinite dimensional Lie algebra satisfying
	\beq
	[\alpha_k,\alpha_\ell] \= (k-\ell) \, \alpha_{k+\ell} \, ,\quad \forall \, k,\ell\geq -1\,,
	\eeq
	where $[\,,\,]$ denotes the Lie bracket of the Lie algebra, define 
	\beq
	\tilde \alpha_{m} \:= - \, \sum_{k\geq -1} \, \frac{(-m)^{k+1}}{(k+1)!} \, \alpha_k,\quad m\geq 0\,. 
	\eeq
	Then
	\beq\label{viraHtilde1216}
	\left[ \tilde \alpha_{m},\, \tilde \alpha_{n} \right] \= (m-n) \, \tilde \alpha_{m+n}, \quad \forall \, m,n\geq 0\,. 
	\eeq
\end{lemma}
\begin{proof}
	\begin{align}
		\bigl[ \tilde \alpha_{m},\, \tilde \alpha_{n} \bigr] 
		\=& \sum_{k_1,k_2=-1}^\infty  \frac{(-m)^{k_1+1} \, (-n)^{k_2+1}}{(k_1+1)! \, (k_2+1)!} (k_1-k_2)\, \alpha_{k_1+k_2} \nn\\
		\=& \sum_{k=-1}^\infty \Biggl( \sum_{k_1+k_2=k\atop k_1\geq 0,k_2\geq -1} 
		\frac{(-m)^{k_1+1} \, (-n)^{k_2+1}}{k_1! \, (k_2+1)!}  
		- \sum_{k_1+k_2=k \atop k_1\geq-1,k_2\geq 0 } \frac{(-m)^{k_1+1} \, (-n)^{k_2+1}}{(k_1+1)! \, k_2!} \Biggr) \, \alpha_k \nn\\
		\=& \sum_{k=-1}^\infty \, \frac{(-1)^k}{(k+1)!} \, \bigl(  m \, (m+n)^{k+1} - n \, (m+n)^{k+1} \bigr) \, \alpha_k \= (m-n)\,  \tilde \alpha_{m+n}\,. \nn
	\end{align}
	The lemma is proved.
\end{proof}

\begin{prop}
We have
\begin{align}
\widetilde{L}_0^{\rm H}\bigl(\e^{-1} \, \tilde {\bf t}, \e \, \p/\p {\bf t}\bigr) \= & - \, \sum_{i\geq1} \, \tilde t_i \, \frac{\p }{\p t_{i-1}} \,-\, \frac{t_0^2}{2 \, \e^2} \+ \frac1{16}\,, \label{modifiedL0}\\
\widetilde{L}_1^{\rm H}\bigl(\e^{-1} \, \tilde {\bf t}, \e \, \p/\p {\bf t}\bigr) \= & \sum_{i\geq 0} \, \sum_{j=0}^{i+1} \, \sum_{r=0}^{i+1-j} \, \binom{i+1}{j+r} \, \frac{(-1)^{i-j-r}}{2^{r}} \, \tilde t_j \, 
\frac{\p}{\p t_i}  \nn\\
& \,-\, \frac{\epsilon^2}{8} \, \sum_{i,j\geq0} \, \Bigl(-\frac{1}{2}\Bigr)^{i+j} \, \frac{\p^2}{\p t_i \, \p t_j }
\,-\, \frac{t_0^2}{2 \, \epsilon^2} \+ \frac{1}{8} 
\,, \label{modifiedL1}\\
\widetilde{L}_2^{\rm H}\bigl(\e^{-1} \, \tilde {\bf t}, \e \, \p/\p {\bf t}\bigr) \= & - \, \frac12 \, \sum_{i_1,i_2,i_3,i,j\geq 0 \atop i_2 \leq i_1+i+j+i_3 \leq i_2+1} \, (-1)^{i+j+i_1+i_2+1} \, \binom{i_2+1}{i_3} \, 3^j \, 2^{i_3-i-j} \, \tilde t_{i_1} \, \frac{\p }{\p t_{i_2}} \nn\\
& - \, \frac12 \, \sum_{i_1,i_2,i_3,i,j\geq 0 \atop 
i_1 \le i+j+i_2+i_3\le i_1+1} \, (-1)^{i+j} \, \binom{-i_2}{i_3} \, 3^j \, 2^{i_3-i-j} \, \tilde t_{i_2} \, \frac{\p }{\p t_{i_1}} \nn\\
& - \, \frac{ \e^2 }2 \, \sum_{ i_1,i_2,i_3,i,j \ge0 \atop i_1+i_2+1 \leq i+j+i_3 \leq i_1+i_2+2} \, (-1)^{i+j+i_2+1} \, 3^j \, 2^{i_3-i-j} \, \binom{i_2+1}{i_3} \, \frac{\p^2 }{\p t_{i_1} \p t_{i_2}}  \nn\\
& - \, \frac{t_0^2}{2 \, \e^2} \+ \frac3{16}\,. \label{modifiedL2}
\end{align}
\end{prop}

\begin{proof}
It follows from \eqref{LkHgivental}, \eqref{wlmdef} that 
the operators $\widetilde L_m^{\rm H}$, $m\geq 0$, have the following expressions:
\beq \label{wlm}
\widetilde L_m^{\rm H} \bigl( \e^{-1} \, \tilde {\bf t}, \e \, \p/\p{\bf t}\bigr) \= 
 \widehat{\Psi}_m\big|_{q_i\mapsto \tilde t_i, \, \p_{q_i} \mapsto \p_{t_i},\, i\geq 0}  \+ \frac{m+1}{16}\,,
\eeq
where 
\beq
\Psi_m(z) \:= - \, \sum_{k=-1}^\infty \, \frac{(-m)^{k+1}}{(k+1)!} \, \Phi_k(z)\,.
\eeq

Let us now prove~\eqref{modifiedL0}--\eqref{modifiedL2} one by one.  For $m=0$, we have
\beq
\widetilde L_0^{\rm H} \bigl( \e^{-1} \, \tilde {\bf t}, \e \, \p/\p{\bf t}\bigr) \=  - \, L_{-1}^{\rm H}\bigl(\e^{-1} \, \tilde {\bf t}, \e \, \p/\p {\bf t}\bigr) \,.
\eeq

For $m=1$, 
$$
\Psi_1(z) \= - \, \Phi(z) \, z^{3/2} \, e^{\p_z} \circ z^{-1/2} \circ \Phi^{-1}(z) \= - \, 
z^{3/2} \, (z+1)^{-1/2} \, \frac{\Phi(z)}{\Phi(z+1)} \, e^{\p_z}.
$$
Note that the following identity
\beq\label{ratio}
\frac{\Phi(z+1)}{\Phi(z)} \= \frac{z+\frac12}{\sqrt{z \, (z+1)}}
\eeq
is proved in~\cite{AIKZ}. Using this identity we get 
\beq
\Psi_1(z) \= - \, \frac{ z^2} {z+\frac12} \, e^{\p_z} \, .  
\eeq
Then by a straightforward computation similar to the one given in~\cite{DLYZ2} we arrive at formula~\eqref{modifiedL1}. 
For $m=2$, we notice the identity 
$$
\frac{\Phi(z+2)}{\Phi(z)} \= \frac{\bigl(z+\frac32\bigr) \, \bigl(z+\frac12\bigr)}{(z+1) \, \sqrt{z \, (z+2)}} \,.
$$
Similarly as above we get 
\beq
\Psi_2(z) \= - \, z^{3/2} \, (z+2)^{-1/2} \, \frac{\Phi(z)}{\Phi(z+2)} \, e^{2\p_z}
\= - \, z^2 \, \frac{z+1}{(z+\frac12) \, (z+\frac32)}  \, e^{2\p_z}.
\eeq
Formula~\eqref{modifiedL2} then follows. The theorem is proved.
\end{proof}

Let us continue the preparations. The following lemma establishes the relationship between the operator $L_{\rm dilaton} $ and 
the operator $\widetilde L^{\rm H}_{\rm dilaton}$ defined in the previous section. 
\begin{lemma} \label{dilatonlemma}
We have 
\begin{align}\label{dilatonid}
e^{-\frac{A(x,T)}{\hbar^2}} \circ L_{\rm dilaton} \circ e^{\frac{A(x,T)}{\hbar^2}} \= 
\widetilde L^{\rm H}_{\rm dilaton} \big|_{\epsilon=\hbar\,\sqrt{-4}} \,.
\end{align}
\end{lemma}
\begin{proof}
The identity follows immediately from $\Bigl[L_{\rm dilaton}, e^{\frac{A(x,T)}{\hbar^2}}\Bigr]=0$ and $\bigl[L_{\rm dilaton}, t_i\bigr]=\tilde t_i$.
\end{proof}

Similarly, we are to establish the relationships between the operators $L_m$ and the operators $\widetilde L_m^{\rm H}$. 
To this end, 
the following lemma would be helpful.

\begin{lemma}\label{lem:vira-bgw-hodge}
	For $i\geq0$, let
	\begin{align}
		X_i \= &\delta_{i,0}\, x \,- \,\sum_{a\geq 0}\,\Big(-\frac{2a+1}{2}\Big)^i \, \frac{\widetilde {T}_{2a+1} }{(a+1)!} \, , \\
		Y_i\= &\frac{2}{3}\,\delta_{i,0}\, x \,- \, \sum_{a\geq 0} \,\Big(-\frac{2a+1}{2}\Big)^i \,\frac{\widetilde {T}_{2a+1} }{(a+2)!} \, ,
	\end{align}
	then 
	\begin{align}
		X_{i} \=&\frac{1}{2^{i}} \, X_0 \,-\, \sum_{j=0}^{i-1}\,\frac{1}{2^{i-j}}\, \tilde t_j \,, \\
		Y_{i} \=&\frac{3^i}{2^i} \, Y_0\,-\,\frac{3^{i}-1}{2^{i}}\,X_0\+\sum_{j=0}^{i-2} \,\frac{3^{i-1-j}-1}{2^{i-j}}\,\tilde t_j \,.
	\end{align}
\end{lemma}
\begin{proof}
By observing that $Y_{i+1} = \frac{3}{2}\,Y_i-X_i$ and $X_{i+1} = \frac12 \,X_i-\frac12\,\tilde t_i$.
\end{proof}

\begin{prop}\label{conjlm} The following identities are true:
\begin{align}
e^{-\frac{A(x,T)}{\hbar^2}} \circ L_m \circ e^{\frac{A(x,T)}{\hbar^2}} \= 
\widetilde L_m^{\rm H} \big|_{\epsilon=\hbar\,\sqrt{-4}}\,,\quad m\geq0\,. \label{conjeq1215}
\end{align}
\end{prop}
\begin{proof}
Because of the Virasoro commutation relations~\eqref{viralmbgw1216} and~\eqref{viraHtilde1216}, we just need to prove the $m=0,1,2$ cases.
	For $m=0$,
	\begin{align}
		& e^{-\frac{A(x,T)}{\hbar^2}} \circ L_0 \circ e^{\frac{A(x,T)}{\hbar^2}} \nn\\
		\=& \sum_{a\geq 0}\,\frac{2a+1}{2}\,\widetilde T_{2a+1}\frac{\p }{\p T_{2a+1}}
		\+\frac{1}{16}\+\frac{x^2}{8\, \hbar^2} \+\frac{1}{\hbar^2}\,\sum_{a\geq 0}\,\frac{2a+1}{2}\,\widetilde T_{2a+1}\,\frac{\p A}{\p T_{2a+1}} \nn\\
		\=& \sum_{i\geq 0}\,\sum_{a\geq 0} \,2\,\Big(-\frac{2a+1}{2}\Big)^{i+1}\,\frac{\widetilde T_{2a+1}}{a!}\, \frac{\p }{\p t_i}
		\+\frac{1}{16}\+\frac{x^2}{8\, \hbar^2} \nn\\
		& \+ \frac{1}{\hbar^2}\,\sum_{a, b \geq 0}\,\frac{\widetilde T_{2a+1} \widetilde T_{2b+1}}{2\, a!\,b!}\,-\,  \frac{1}{\hbar^2}\,\sum_{a\geq 0}\, \frac{x\, \widetilde T_{2a+1}}{2\,a!}\nn\\
		\=&-\sum_{i\geq 0}\,\tilde t_{i+1}\,\frac{\p }{\p t_i}\+\frac{1}{16} \+\frac{t_0^2}{8\,\hbar^2}\nn\\
		\=& \widetilde L_{0}^{\rm H}\big|_{\epsilon=\hbar\sqrt{-4}}. \nn
	\end{align}
For $m=1$,
	\begin{align}
	& e^{-\frac{A(x,T)}{\hbar^2}} \circ L_1 \circ e^{\frac{A(x,T)}{\hbar^2}} \nn\\
	\=& \sum_{a\geq 0}\,\frac{(2a+1)(2a+3)}{4}\,\widetilde T_{2a+1}\frac{\p }{\p T_{2a+3}}
	\+\frac{\hbar^2}{8}\,\frac{\p^2}{\p T_1^2} \+\frac{1}{\hbar^2}\,\sum_{a\geq 0}\,\frac{(2a+1)(2a+3)}{4}\,\widetilde T_{2a+1}\frac{\p A}{\p T_{2a+3}}\nn\\
	&\+\frac{1}{8\,\hbar^2}\,\frac{\p A}{\p T_1}\,\frac{\p A}{\p T_1} \+\frac{1}{8}\,\frac{\p^2 A}{\p T_1^2} \+\frac{1}{4}\,\frac{\p A}{\p T_1}\,\frac{\p }{\p T_1} \nn\\
	\=&({\rm I})\+({\rm II})\+({\rm III}) \,, \label{781215}
\end{align}
where
\begin{align}
	({\rm I})\:=& \frac{1}{\hbar^2}\,\sum_{a\geq 0}\,\frac{(2a+1)(2a+3)}{4}\, \widetilde T_{2a+1}\,\frac{\p A}{\p T_{2a+3}}
	\+\frac{1}{8\,\hbar^2}\,\frac{\p A}{\p T_1}\,\frac{\p A}{\p T_1} \+\frac{1}{8}\,\frac{\p^2 A}{\p T_1^2}\,, \\
	({\rm II})\:=& \sum_{a\geq 0}\,\frac{(2a+1)(2a+3)}{4}\,\widetilde T_{2a+1}\, \frac{\p }{\p T_{2a+3}} \+\frac{1}{4}\,\frac{\p A}{\p T_1}\,\frac{\p }{\p T_1}\,, \\
	({\rm III})\:=& \frac{\hbar^2}{8}\,\frac{\p^2}{\p T_1^2}\, .
\end{align}
By a direct computation, 
\beq
({\rm I})\=\frac{1}{8}\+\frac{t_0^2}{8\,\hbar^2},\qquad
({\rm III})\=\frac{\hbar^2}{2}\,\sum_{i,j\geq 0}\,\Big(-\frac{1}{2}\Big)^{i+j}\,\frac{\p^2}{\p t_i\p t_j} \,, \label{821215}
\eeq
and
\begin{align}
	({\rm II})
	\= &\sum_{i\geq 0}\,\bigg(\sum_{a\geq 0}\,(2a+1)\,\frac{\widetilde T_{2a+1}}{(a+1)!} \, \Bigl(-\frac{2a+3}{2}\Bigr)^{i+1}
	\+\frac{1}{2}\,(-\frac{1}{2})^i\,\Big( x-\sum_{a\geq 0}\frac{\widetilde {T}_{2a+1}}{(a+1)!}\Big)\bigg)\, \frac{\p }{\p t_i}\nn\\
	\= &\sum_{i\geq 0}\,\bigg(-2\sum_{j=0}^{i+1}\,\binom{i+1}{j}(-1)^{i+1-j}\,\sum_{a\geq 0} \frac{\widetilde {T}_{2a+1}}{(a+1)!}\,\Bigl(-\frac{2a+1}{2}\Bigr)^{j+1}
	\+\frac{1}{2}\,\Bigl(-\frac{1}{2}\Bigr)^i\,X_0 \bigg)\, \frac{\p }{\p t_i}\nn \\
	\= &\sum_{i\geq 0}\,\bigg(2\, \sum_{j=0}^{i+1}\,\binom{i+1}{j}(-1)^{i+1-j} \, X_{j+1}
	\+\frac{1}{2}\,\Bigl(-\frac{1}{2}\Bigr)^i \, X_0\bigg)\, \frac{\p }{\p t_i} \nn\\
	\= &\sum_{i\geq 0} \sum_{j=0}^{i+1} \,\bigg(\sum_{r=0}^{i+1-j}\,\binom{i+1}{j+r}\,(-1)^{i-j-r}\,\frac{1}{2^{r}}\bigg)\,\tilde t_j\,\frac{\p}{\p t_i}\,. \label{831215}
\end{align}
Here we have used Lemma \ref{lem:vira-bgw-hodge}.  
From~\eqref{781215}, \eqref{821215}, \eqref{831215} we obtain the validity of the $m=1$ case of~\eqref{conjeq1215}.
In a similar way, for $m=2$ we can obtain 
\begin{align}
	e^{-\frac{A(x,T)}{\hbar^2}} \circ L_2 \circ e^{\frac{A(x,T)}{\hbar^2}}
	 \=&\sum_{i\geq 0} \sum_{j=0}^{i+1} \, \bigg(\sum_{r=j}^{i+1}\,\binom{i+1}{r}\,(-2)^{i-r}\,\frac{3^{r-j}+1}{2^{r-j}}\bigg)\, \tilde t_j\,\frac{\p}{\p t_i} \nn\\
	 & \+\frac{\hbar^2}{2}\,\sum_{i,j\geq 0}\, \Bigl(-\frac{1}{2}\Bigr)^{i+j}\,3^{j+1}\,\frac{\p^2}{\p t_i\p t_j} \+ \frac{3}{16}\+\frac{t_0^2}{8\,\hbar^2}.
\end{align}
One can check that this equals $\widetilde{L}_2^{\rm H}$ given in~\eqref{modifiedL2}. 
The proposition is proved.
\end{proof}

We are ready to prove Theorem~\ref{main}.

\begin{proof}[Proof of Theorem~\ref{main}] 
By using the degree-dimension counting~\eqref{ddhodge} and~\eqref{deftT}, we find that the expression
$\mathcal{H}\bigl({\bf t}(x, {\bf T}); \hbar\,\sqrt{-4}\bigr)$ is a well-defined element 
in $\hbar^{-2}\CC[[\hbar^2]][[x+2]][[{\bf T}]]$, therefore, 
$e^{\frac{A(x,{\bf T})}{\hbar^2}} Z_{\rm H}\bigl({\bf t}(x, {\bf T}); \hbar\,\sqrt{-4} \bigr)$ 
is a well-defined element in  
$\CC((\hbar^2))[[x+2]][[{\bf T}]]$. It follows from Proposition~\ref{conjlm} that
\beq\label{viraaZ1216}
L_m \Bigl(e^{\frac{A(x,{\bf T})}{\hbar^2}} Z_{\rm H}\bigl({\bf t}(x, {\bf T}); \hbar\,\sqrt{-4} \bigr)\Bigr) \= 0 \,, \quad m\geq 0 \,, 
\eeq
which are the same as the linear equations~\eqref{eqn:vira-gBGW-x} for the power series $Z(x,{\bf T};\hbar)$. 
It then follows from Proposition~\ref{cajZ} that $e^{\frac{A(x,{\bf T})}{\hbar^2}} Z_{\rm H}\bigl({\bf t}(x, {\bf T}); \hbar \,\sqrt{-4} \bigr)$
and $Z(x,{\bf T};\hbar)$ could only differ by multiplying by a non-zero element in $\CC((\hbar^2))[[x+2]]$. More precisely, if 
we denote 
\beq 
R(x,\hbar) \:= \frac{A(x,{\bf 0})}{\hbar^2} \+ \cH \bigl({\bf t}(x, {\bf 0}); \hbar\,\sqrt{-4} \bigr) \;\in\; \hbar^{-2}\CC[[x+2]][[\hbar^2]] \,,
\eeq 
then it remains to show that $R(x,\hbar)=B(x,\hbar)$.
 Write 
 \beq
 R(x,\hbar) \;=:\; \sum_{g\geq0} \, \hbar^{2g-2} \, R_g(x) \,.
 \eeq
 
Let us first compute $R_{0}(x)$. Let $v({\bf t})=\frac{\p^2\cH_0({\bf t})}{\p t_0^2}$. It is well known (cf.~e.g.~\cite{DLYZ1}) that 
\begin{align}
& v({\bf t})\=t_0\+\sum_{k\geq 2}\,\frac{1}{k}\,\sum_{\substack{i_1,\cdots, i_k\geq 0\\ i_1+\cdots + i_k=k-1}} 
\, \frac{t_{i_1} }{i_1!}\,\cdots \,\frac{t_{i_k}}{i_k!} \,, \label{explicitv1216}\\
& \cH_0({\bf t}) \= \frac12 \, \sum_{i,j\geq0} \, \tilde t_i \, \tilde t_j \, \frac{v({\bf t})^{i+j+1}}{i! \, j! \, (i+j+1)} \,.
\end{align}
Since $t_i(x,{\bf 0})=\delta_{i,0} \, x+\delta_{i,1}-(-\frac{1}{2})^{i-1}$, we find
 \begin{align}
 v\big({\bf t}(x,{\bf 0})\big)
\=&-\,2\,\log\Big(-\frac{x}{2}\Big) \,, \label{iniv1216}
 \end{align}
and therefore, 
\begin{align}
	R_0(x)\=& A(x,{\bf 0})\,-\,\frac{1}{8}\,\sum_{i,j\geq 0}\,\tilde t_{i}(x,{\bf 0})\,\tilde t_{j}(x,{\bf 0})\, \frac{v\big({\bf t}(x,{\bf 0})\big)^{i+j+1}}{i!\,j!\,(i+j+1)}
	\= \frac{x^2}{4} \, \log \Bigl(-\frac{x}{2}\Bigr) \,-\, \frac38 \, x^2.
\end{align}

Let us now compute $R_1(x)$. It is known e.g. in~\cite{DLYZ1,DLYZ2,DY1} that 
\beq\label{cH1exp}
\cH_1({\bf t}) \= \frac1{24} \, \log\bigl(v_{t_0}({\bf t})\bigr) \,-\, \frac1{16} \, v({\bf t})\,.
\eeq
It then follows from \eqref{iniv1216} and the fact $\p_x=\p_{t_0}$ that 
\beq
R_1(x) 
\=\frac1{12} \, \log \Big(-\frac{x}{2}\Big)  \,.
\eeq

Using Lemma~\ref{dilatonlemma} we find that for $g\geq 2$ the element $R_g(x)$ has the form
$R_g(x) = r_g/x^{2g-2}$ for some $ r_g\in \CC$.
Define 
\beq
w \= w(x, {\bf T};\hbar) \:= \Bigl( 2 \,-\, \Lambda^{\frac12} \,-\, \Lambda^{-\frac12} \Bigr) \, \bigl(\cH\bigl({\bf t}(x, {\bf T}); \hbar \, \sqrt{-4}\bigr)\bigr) \,. 
\eeq
where $\Lambda:=\exp (2 \, \hbar \, \sqrt{-2} \, \p_x)$. Noticing that 
\beq
\frac{\p}{\p T_1} \= \sum_{i\geq0} \, \frac{1}{(-2)^{i-1}} \, \frac{\p }{\p t_i} \,,
\eeq
we know from the Hodge--FVH correspondence~\cite{LYZZ, LZZ} that $w(x, {\bf T};\hbar)$ satisfies the following equation:
\beq\label{Bogo1216}
- \, \frac14 \, \frac{\p w}{\p T_1} \= \frac{1}{\hbar \, \sqrt{-2}} \, \frac{2-\Lambda^{\frac12}-\Lambda^{-\frac12}}{\Lambda^{-\frac12}-\Lambda^{\frac12}} \, \bigl(e^{w}\bigr) \,. 
\eeq
Using the $m=0$ case of~\eqref{viraaZ1216} we find that 
\beq\label{ini3rd1216}
\hbar^2 \, \frac{\p^3 \cH\bigl({\bf t}(x, {\bf T}); \hbar \, \sqrt{-4}\bigr)}{\p x \p x \p T_1} \bigg|_{{\bf T}=0} \;\equiv\; \frac12 \,. 
\eeq
Denote $W(x,\hbar)=w(x, {\bf 0};\hbar)$. Using the expressions of $R_0(x)$ and $R_1(x)$ we know that  
$W(x,\hbar) = \log \bigl(- \, \frac{x}2 \bigr) + {\rm O}\bigl(\hbar^4\bigr)$. 
 It follows from~\eqref{Bogo1216} and~\eqref{ini3rd1216} that 
\beq
  \p_x \Bigl(e^{W(x,\hbar)}\Bigr) \=  - \, \frac12 \,.
\eeq
We therefore conclude that, for $g\geq2$, 
$r_g = \frac{(-1)^g \, 2^{g-1} \, B_{2g}}{2g \, (2g-2)} $. The theorem is proved. 
\end{proof} 

Since for any fixed parameter~$x$, the power series $Z(x,{\bf T};\hbar)$ is a 
 tau-function for the KdV hierarchy~\cite{MMS} (cf.~\cite{Dickey, DYZ}) we obtain immediately from 
the Hodge-BGW correspondence~\eqref{mainid} the following corollary. 
\begin{cor} \label{corkdv}
For any fixed~$x$, the power series $e^{\frac{A(x, {\bf T})}{\hbar^2}} Z_{\rm H}\bigl({\bf t}(x,{\bf T}); \hbar \, \sqrt{-4} \bigr)$ is a
 tau-function for the KdV hierarchy~\eqref{KdVThbar} associated to the solution specified by the initial value 
 $$\frac{x^2}4 \, \frac{1}{(1-T_1)^2} \+ \frac{\hbar^2}{8} \, \frac{1}{(1-T_1)^2}\,.$$ 
\end{cor}

\section{ELSV-like formula for the generalized BGW correlators}\label{section4}
In this section, we prove Proposition~\ref{thm:gbgw-elsv}. 

\begin{proof}[Proof of Proposition~\ref{thm:gbgw-elsv}]
From the definition~\eqref{cubichodgeint} of the cubic Hodge free energy  
and using~\eqref{deftT} 
we obtain that 
\begin{align}
 &\cH\bigl({\bf t}(x, {\bf T});\hbar\,\sqrt{-4}\bigr)\,\nn\\
\=& \sum_{g\geq0}\, (-4)^{g-1}\,\hbar^{2g-2}\,\sum_{n\geq 0}\,\frac{1}{n!} \,\int_{\overline{\mathcal{M}}_{g,n}} \Lambda(-1)^2 \, \Lambda\Bigl(\frac{1}{2}\Bigr)  \,\prod_{p=1}^n \,\biggl(x+2+\frac{\psi_p^2}{2+\psi_p} -2\sum_{a\geq 0}\frac{T_{2a+1}/a!}{1+\frac{2a+1}{2}\psi_p}\biggr)\, \nn \\
  \=& \sum_{g\geq0}\, (-4)^{g-1}\,\hbar^{2g-2}\,\sum_{n\geq 0}\,\frac{1}{n!}\, \int_{\overline{\mathcal{M}}_{g,n}}  \Lambda(-1)^2 \, \Lambda\Bigl(\frac{1}{2}\Bigr) \,
  \sum_{l=0}^{n}\,\binom{n}{l}\prod_{p=l+1}^{n}\, \biggl( x+2+\frac{\psi_p^2}{2+\psi_p}\biggr) \nn\\
  & \times \sum_{k_1,\cdots, k_l}\,(-2)^l \, \prod_{p=1}^l \, \frac{T_{2a_p+1}/a_p!}{1+\frac{2a_p+1}{2}\psi_p} \,. \label{TaylorHT}
\end{align}

For $g,l\geq0$ and for $a_1,\dots,a_l\geq0$, by comparing the coefficients of $\hbar^{2g-2} T_{2a_1+1}\cdots T_{2a_l+1}$ 
of the logarithms of both sides of~\eqref{mainid} and using \eqref{genusexpansion}, \eqref{defA1215}, \eqref{TaylorHT} we get
\begin{align}
	&c_g(a_1,\dots,a_l) \, x^{2|a|-2g+2} \, \delta_{l\geq1} \+  
	\biggl(\frac{x^2}{4} \, \log \Bigl(-\frac{x}{2}\Bigr) - \frac38 \, x^2\biggr) \, \delta_{g,0} \, \delta_{l,0} \,  \+ \frac1{12} \, \log \Big(-\frac{x}{2}\Big) \,  \delta_{g,1} \, \delta_{l,0} \nn\\
	& \+ \frac{1}{x^{2g-2}}\frac{(-1)^g \, 2^{g-1} \, B_{2g}}{2g \, (2g-2)} \, \delta_{g\geq2} \, \delta_{l,0} \nn\\
	\= &\frac{(-1)^{g-1+l} \, 2^{2g-2+l}}{\prod_{p=1}^l a_p!} \, \sum_{k\geq  0} \, \frac{(x+2)^k}{k!} \, \sum_{n'\geq 0} \, \frac{1}{n'!\,2^{n'}} \,\nn\\
	& \, \times \int_{\overline{\mathcal{M}}_{g,l+k+n'}}  \Lambda(-1)^2 \, \Lambda\Bigl(\frac{1}{2}\Bigr) \,
	\prod_{p=1}^l \,\frac{1}{1+\frac{2a_p+1}{2}\psi_p}\,
	\prod_{q=l+k+1}^{l+k+n'} \,\frac{\psi_q^2}{1+\frac{1}{2}\psi_q} \nn\\
	& \+\delta_{g,0} \, \delta_{l,2} \, \frac{1}{a_1! \, a_2! \, (a_1+a_2+1)} 
	\,-\, \delta_{g,0} \, \delta_{l,1} \, \frac{1}{a_1! \, (2a_1+1)} \, \bigg(x+\frac{1}{a_1+1}\bigg)\,. 
	\label{cghodge}
\end{align}

For $l\geq 1$, by comparing the coefficients of $(x+2)^{2|a|-2g+2}$ of both sides of~\eqref{cghodge} we obtain 
\begin{align}
	&c_{g}(a_1,\cdots,a_l) 
	\= \frac{(-1)^{g-1+l} \, 2^{2g-2+l}}{(2|a|-2g+2)! \, \prod_{p=1}^la_p!} \, \sum_{n'\geq0} \, \frac{1}{n'! \, 2^{n'}} \,\nn\\
	&\qquad\times \int_{\overline{\mathcal{M}}_{g,l+2|a|-2g+2+n'}} \Lambda(-1)^2 \, \Lambda\Bigl(\frac{1}{2}\Bigr) 
	 \, \prod_{p=1}^l\, \frac{1}{1+\frac{2a_p+1}{2}\psi_p}
	 \, \prod_{q=l+2|a|-2g+2+1}^{l+2|a|-2g+2+n'}\, \frac{\psi_q^2}{1+\frac{1}{2}\psi_q} \,.
\end{align}
Similarly, for $l=0$ and $g\ge2$, by comparing the coefficients of $(x+2)^0$ of both sides of~\eqref{cghodge} we obtain  
\begin{align}
\frac{(-1)^g \, 2^{g-1} \, B_{2g}}{2g \, (2g-2)} \= & (-1)^{g-1} \, 4^{2g-2} \, 
\sum_{n'\ge0} \, \frac{1}{n'! \, 2^{n'}} \, \int_{\overline{\mathcal{M}}_{g,n'}}  \Lambda(-1)^2 \, \Lambda\Bigl(\frac{1}{2}\Bigr) \, 
	\prod_{q=1}^{n'} \frac{\psi_q^2}{1+\frac{1}{2}\psi_q}\,.
\end{align}

Consider the forgetful map
\beq
f: \overline{\mathcal M}_{g,n+n'} \to \overline{\mathcal M}_{g,n}
\eeq
forgetting the last $n'$ marked points. It is known from~\cite{AC, KMZ, LX, MZ} that 
\begin{align}
& \sum_{n'\geq 1} \, \frac{1}{n'!} \, \sum_{d_1,\cdots,d_{n'}\geq 1} \, 
 f_{*}\bigg(\Lambda(-1)^2 \, \Lambda\Bigl(\frac{1}{2}\Bigr) 
	 \, \prod_{p=1}^l\, \frac{1}{1+\frac{2a_p+1}{2}\psi_p}\,\prod_{r=1}^{n'}\psi_{n+r}^{d_r+1}\bigg)\, \prod_{r=1}^{n'} b_{d_r}({\bf s}) \nn\\
& \=  e^{\sum_{d\geq 1}s_d\kappa_d}  \, \Lambda(-1)^2 \, \Lambda\Bigl(\frac{1}{2}\Bigr) 
	 \, \prod_{p=1}^l\, \frac{1}{1+\frac{2a_p+1}{2}\psi_p}\,,\label{eqn:t-kappa}
\end{align}
where ${\bf s}=(s_1,s_2,\dots)$ is an infinite vector of indeterminates and $b_d({\bf s})$ are polynomials of ${\bf s}$ defined via 
\beq
\exp\biggl(-\sum_{d\geq 1} \, s_d \, z^d\biggr) \= 1 \, - \, \sum_{d\geq 1} \, b_d({\bf s}) \, z^d \,.
\eeq
(This was also used in~\cite{GF} to derive the ELSV-like formula for even GUE correlators from the Hodge-GUE correspondence.)
By solving the following system of equations
\beq
-b_d({\bf s}) \= \Bigl(-\frac{1}{2}\Bigr)^d\,, \quad d\geq 1 \,, 
\eeq
we get $s_d=-\frac{1}{d}(-\frac{1}{2})^d$, thus we obtain identities~\eqref{elsvgen1215} and~\eqref{elsvgen21215}. 
\end{proof}

In the above proof we get identities~\eqref{cghodge} (which hold in $\CC[[x+2]]$), 
and obtained Proposition~\ref{thm:gbgw-elsv} from these identities.  We note that one can also get from~\eqref{cghodge} 
infinitely many other identities by comparing coefficients of powers of $x+2$; for example, for $l\geq1$ and $a_1,\dots,a_l\geq0$,  
one can obtain vanishing identities (cf.~also~\cite{DY3}) 
for certain combinations of intersection numbers by taking 
coefficients of $(x+2)^k$, for $k\geq2|a|-2g+3$.

We hope that the ELSV-like formula~\eqref{eqn:bgw-elsv} for the BGW correlators could be helpful for 
understanding Norbury's $\Theta$-class integrals~\cite{KN, N}.

\medskip

\noindent {\bf Acknowledgements.} 
One of the authors Q.Z. would like to thank Jian Zhou and Shuai Guo for very helpful discussions. 
Part of the work of Q.Z. was done during his visits to USTC; he thanks USTC for warm hospitality. 
The work is partially supported by the National Key R and D Program of China 2020YFA0713100, 
and by NSFC No.~12061131014. 

\medskip

\begin{appendix}
\section{Three more applications}\label{appa}
In this appendix, we give three more applications of the Hodge-BGW correspondence. 

Introduce a power series $Q=Q(x,\bbT)\in \CC[[x+2]][[\bbT]]$ by 
\beq\label{vQ1215}
Q(x,\bbT) \= \exp\biggl(-\frac{v({\bf t}(x,\bbT))}{2}\biggr) \,.
\eeq

\begin{lemma}
The power series $Q$ is the unique solution in $\CC[[x+2]][[\bbT]]$ to the initial value problem of the following integrable hierarchy:
\begin{align}
&\frac{\p Q}{\p T_{2a+1}} \= - \, \frac{2}{a!} \, Q^{2a+1} \, \frac{\p Q}{\p x} \,, \quad q\geq 0\,,\label{flow-q}\\
&Q(x,{\bf 0}) \= - \, \frac{x}{2}\,.  \label{ini-Q1215}
\end{align}
Moreover, $Q$ satisfies the equation
	\beq\label{euler-lagrange-Q}
	Q \= - \, \frac{x}{2} \+ \sum_{a\geq0} \, {T}_{2a+1} \, \frac{Q^{2a+1}}{a!} \,.
	\eeq
\end{lemma}
\begin{proof} 
Equalities~\eqref{flow-q} follow from~\eqref{vQ1215} and the dispersionless KdV hierarchy
\beq v_{t_i}\= \frac{v^i}{i!} \, v_{t_0}\,; \eeq
cf.~\eqref{explicitv1216}.
The property~\eqref{ini-Q1215} follows from the already proved formula~\eqref{iniv1216}, and 
equality~\eqref{euler-lagrange-Q} is then an elementary exercise.
\end{proof}

We note that solution in $\CC[[x+2]][[\bbT]]$ to equation~\eqref{euler-lagrange-Q} is also unique. 
Furthermore, from \eqref{flow-q}, \eqref{ini-Q1215} we see that this unique solution~$Q$ actually belongs to 
$\mathbb C[x][[{\bf T}]]\subset\CC[[x+2]][[\bbT]]$.

The following proposition on $\F_0(x, \bbT)$ was proved in a joint work of Jian Zhou and the second-named author of the present paper (cf.~\cite{Z}) by constructing  
the constitutive relation \cite{DW,ZZ} of the generalized BGW model and by solving its Virasoro constraints. 
Here we give a new and short proof as an application of the Hodge-BGW correspondence~\eqref{mainid}.

\begin{prop}\label{prop-bgwf0}
The genus zero part of the generalized BGW free energy $\F_0(x, \bbT)$ admits the expression (cf.~\cite{Du1})
\begin{align}\label{dubrovin-f0}
\F_0(x, \bbT) \= & \frac12 \, \sum_{a,b\geq0} \, \widetilde {T}_{2a+1} \, \widetilde {T}_{2b+1} \, \frac{Q^{2a+2b+2}}{a!\,b!\,(a+b+1)}
		\,-\, x \, \sum_{b\geq0} \, \widetilde {T}_{2b+1} \, \frac{Q^{2b+1}}{b!\,(2b+1)}  \+  \frac{x^2}{4} \, \log Q \,.
\end{align}
Moreover, the following formula, which is conjectured by Alexandrov~\cite{A},  is true: 
	\begin{align}
	& \F_{0}(x,{\bf T}) \,-\, \biggl(\frac{x^2}{4} \, \log \Bigl(-\frac{x}{2}\Bigr) - \frac38 \, x^2\biggr) \+ \frac{x^2}4 \, \log (1-T_1) \nn \\
	& \= \sum_{\substack{k\geq1, \, j_1,j_2,\dots\geq0 \\ j_1+2j_2 +\cdots = k}} \, 
	\frac{(3j_1+5j_2+\cdots-1)!}{2^{2k+1} \, (1-T_1)^{3j_1+5j_2+\cdots}} \, \frac{x^{2k+2}}{(2k+2)!}\,
	\frac{T_{3}^{j_1}\,T_{5}^{j_2}\,\cdots}{(1!)^{j_1}\, j_1!\, (2!)^{j_2}\, j_2!\,\cdots} \, . \label{f0alex}
	\end{align}
\end{prop}
\begin{proof}
The genus zero part of the dilaton equation~\eqref{eqn:dilaton} reads:
	\beq
	2\,\F_0(x,\bbT)\=x\,\frac{\p \F_0(x,\bbT)}{\p x}\+\sum_{a\geq 0} \, \widetilde T_{2a+1}\,\frac{\p \F_0(x,\bbT)}{\p T_{2a+1}} \,,
	\eeq
from which one deduces that 
	\begin{align}\label{f0dilaton1215}
		\F_0(x,\bbT) \= & \frac12 \, \sum_{a,b\geq0} \, \widetilde {T}_{2a+1} \, \widetilde {T}_{2b+1} \, \frac{\p^2 \F_0(x,\bbT)}{\p T_{2a+1} \,\p T_{2b+1}}
		+ \,x \, \sum_{a\geq0} \, \widetilde {T}_{2a+1} \, \frac{\p^2 \F_0(x,\bbT)}{\p x\, \p T_{2a+1}}  \+  \frac{x^2}{4} \, \frac{\p^2 \F_0(x,\bbT)}{\p x^2} \,.
	\end{align}
	
By the Hodge-BGW correspondence~\eqref{mainid}, we know that 
	\beq
	\F_0(x,\bbT)\= - \, \frac14 \, \cH_0\big({\bf t}(x,\bbT)\big) \+ A(x,\bbT) \,.
	\eeq
Then by using~\eqref{deftT}, \eqref{vQ1215} and the well-known fact that 
\beq
\frac{\p^2\cH_0({\bf t})}{\p t_i \p t_j} \= \frac{v({\bf t})^{i+j+1}}{i! \, j! \, (i+j+1)}\,, \quad i,j\geq0 \,,
\eeq 
with $v({\bf t})$ being defined in~\eqref{explicitv1216},
we obtain that 
\begin{align}
	\frac{\p^2 \F_0(x,\bbT) }{\p T_{2a+1} \p T_{2b+1} }
		& \= \frac{Q^{2a+2b+2}}{a!\, b!\,(a+b+1)}\,, \label{Q1} \\
	\frac{\p^2\F_0(x,\bbT)}{\p x \p T_{2b+1} } & 
\= - \, \frac{Q^{2b+1}}{b! \, (2 b+1)}\,, \label{Q2} \\
	\frac{\p^2\F_0(x,\bbT)}{\p x \p x} & \= \frac12\, \log Q \,. \label{Q3}
	\end{align}
Here $a,b\geq0$.
Substituting \eqref{Q1}--\eqref{Q3} into~\eqref{f0dilaton1215} we obtain~\eqref{dubrovin-f0}. 
 
By using~\eqref{euler-lagrange-Q} and the Lagrangian inversion formula, we find that  
	\beq
	Q\= -\, \sum_{\substack{k, j_1,j_2,\cdots\geq 0\\ j_1+2j_2\cdots =2k}}\, 
	\frac{(3j_1+5j_2+\cdots)!}{2^{2k+1} \, (1-T_1)^{1+3j_1+5j_2+\cdots}} \, \frac{x^{2k+1}}{(2k+1)!}\,
	\frac{T_{3}^{j_1}\,T_{5}^{j_2}\,\cdots}{(1!)^{j_1}\, j_1!\, (2!)^{j_2}\, j_2!\,\cdots} \, .
	\eeq
The proposition is then proved by integrating the $b=0$ identity of~\eqref{Q2} with respect to $x$ and $T_1$, where 
the integration constants can be fixed by using \eqref{Q1}, \eqref{Q2}.
\end{proof}

We note that formula~\eqref{f0alex} was also given in~\cite{OS}. 

\medskip

The next application of the Hodge-BGW correspondence is on a way of calculating 
$\F_g(x,{\bf T})$ with $g\geq1$. Recall from e.g.~\cite{DLYZ1, DY2} that 
the genus~$g$ part of the Hodge free energy $\cH_g({\bf t})$, $g\geq1$, possesses the following jet representation: 
for $g=1$, it is given explicitly by~\eqref{cH1exp}, and for $g\geq2$, there exist functions 
$H_g(v_1,\dots,v_{3g-2})$ of $(3g-2)$ variables, depending polynomially in $v_2,\dots,v_{3g-2}$ and rationally in~$v_1$, such that  
\begin{align}
& \cH_g({\bf t}) \= H_g\biggl(\frac{\p v({\bf t})}{\p t_0},\dots,\frac{\p^{3g-2} v({\bf t})}{\p t_0^{3g-2}}\biggr) \,, \label{jethodge1219}\\
& \sum_{k\geq1} \, k\, v_k \, \frac{\p H_g}{\p v_k} \= (2g-2) \, H_g \,. \label{eqn:dilatonHg}
\end{align}
Moreover, $H_1(v_0,v_1):=\frac1{24} \log (v_1)-\frac1{16} v_0$, and $H_g(v_1,\dots,v_{3g-2})$, $g\ge2$, can be computed by using the 
Dubrovin--Zhang type loop equation~\cite{DLYZ2}. 
Noticing that according to the Hodge-BGW correspondence~\eqref{mainid}, 
\beq\label{HodgeBGWggeq1}
\F_g(x,\bbT) \= (-4)^{g-1} \, \cH_g\big({\bf t}(x, \bbT)\big) \,, \quad g\geq1\,,
\eeq
and using the simple fact from~\eqref{deftT} that $\p_x=\p_{t_0}$, 
we can translate the Dubrovin--Zhang type loop equation for $H_g(v_1,\dots,v_{3g-2})$ to that for higher genus 
parts of the generalized BGW free energy. 
More precisely, we arrive at the following proposition. 
\begin{prop} \label{loopforBGW1219}
Let $u(x,\bbT):= - \,2\, \log Q(x,\bbT)$.  
For any $g\geq1$, the genus~$g$ part of the generalized BGW free energy admits the jet representation: 
for $g=1$, 
\beq\label{eqn:bgwf1-hodgejet}
\F_1(x,\bbT) \= \frac1{24} \log \biggl(\frac{\p u(x,\bbT)}{\p x}\biggr) \,-\, \frac1{16} \, u(x,\bbT)  \,.
\eeq
and for $g\geq2$, there exists a function of $3g-2$ variables $F_g(u_1,\dots,u_{3g-2})$ such that 
\beq\label{eqn:bgwfg-hodgejet}
\F_g(x,\bbT) \= F_g \biggl(\frac{\p u(x,\bbT)}{\p x},\dots,\frac{\p^{3g-2} u(x,\bbT)}{\p x^{3g-2}}\biggr) \,.
\eeq
Moreover, denote $\p=\sum_{j\geq0} u_{j+1} \frac{\p }{\p u_j}$, and let 
\beq
B \= \sqrt{1-\frac{4\,e^{u_0}}{\lambda}}\,, \qquad \Delta F \= \sum_{g\geq 1} \, \hbar^{2g-2} \, F_g(u_1,\dots,u_{3g-2}) \, ,
\eeq
then $\Delta F$ satisfies the following Dubrovin--Zhang type loop equation:
	\begin{align}
		&\sum_{k\geqslant 0} \bigg(\p^k\Bigl(\frac{1}{B^2}\Bigr) \+ 
		\sum_{j=1}^{k}\binom{k}{j} \, \p^{j-1}\Bigl(\frac{1}{B}\Bigr) \, \p^{k-j+1}\Bigl(\frac{1}{B}\Bigr)\bigg) \, \frac{\p \Delta F}{\p  u_k} \nn \\
		&\+ 2\,\hbar^2 \, \sum_{k,l\geqslant 0} \, \p^{k+1}\Bigl(\frac{1}{B}\Bigr) \, \p^{l+1}\Bigl(\frac{1}{B}\Bigr) \, 
		\bigg(\frac{\p^2 \Delta F}{\p  u_k \p  u_l} \+ \frac{\p \Delta F}{\p  u_k} \, \frac{\p \Delta F}{\p  u_l}\bigg)\,\nn\\
		&\+ 2\,\hbar^2\sum_{k\geq 0} \p^{k+2} \Bigl(\frac{1}{8\, B^4}-\frac{1}{4\, B^2}\Bigr) \, \frac{\p \Delta F}{\p  u_k} \+\frac{1}{8\,B^2} \,-\,\frac{1}{16\,B^4} \= 0 \,.
	\end{align}
\end{prop}
\noindent We note that there are also other algorithms for computing $H_g$ (cf.~\cite{DLYZ1,DY2}), 
and therefore for computing~$F_g=(-4)^{g-1}H_g$, $g\geq2$. 

\medskip

Now we consider a third application of the Hodge-BGW correspondence.
Recall that the Gromov--Witten free energy $\mathcal F^{\rm WK}({\bf t};\epsilon)$ of {\it a point} is defined as the logarithm of the partition function $Z_{\rm WK}({\bf t};\epsilon)$ (cf.~\eqref{zwkdef}), and it has the genus expansion:
\beq
\mathcal F^{\rm WK}({\bf t};\epsilon)
\=\sum_{g\geq 0}\, \epsilon^{2g-2}\,\mathcal F^{\rm WK}_{g}({\bf t}).
\eeq
It is known that for every $g\geq 1$, there exists a function of $(3g-2)$ variables 
$F^{\rm WK}_g(z_1,\cdots,z_{3g-2})$ such that 
\beq
\mathcal F^{\rm WK}_g({\bf t})\=F^{\rm WK}_{g}\bigg(\frac{\p v({\bf t})}{\p t_0},\cdots,\frac{\p^{3g-2} v({\bf t})}{\p t_0^{3g-2}}\bigg)\,,
\eeq
where $v({\bf t})=\frac{\p^2 \mathcal F^{\rm WK}_0({\bf t})}{\p t_0^2}$ (cf.~e.g.~\cite{DZ-norm,IZ,ZZ}). 
For $g=1$, the explicit expression~\cite{W} of $F^{\rm WK}_1$ is given by
\beq
F^{\rm WK}_1(z_1)\=\frac{1}{24}\,\log z_1\,.
\eeq
For every $g\geq 2$, $F^{\rm WK}_g(z_1,\cdots,z_{3g-2}) \in \mathbb{Q}\bigl[z_2,\dots,z_{3g-2},z_1^{-1}\big]$.
Dubrovin and Zhang~\cite{DZ-norm} proved that for $g\geq 1$, the functions 
$F^{\rm WK}_g(z_1,\cdots,z_{3g-2})$ are uniquely determined by
\beq\label{eqn:wkdilaton}
\sum_{k=1}^{3g-2}\,k\,z_k\,\frac{\p F^{\rm WK}_g}{\p z_k}\=(2g-2)\,F^{\rm WK}_g \+ \frac{\delta_{g,1}}{24}\,,\qquad g\geq 1
\eeq
and the following {\it loop equation}~\cite{DZ-norm}:
\begin{align}
	&-\sum_{k\geq 0}\, \bigg(\p_1^k\Big(\frac{1}{P^2}\Big) \+\sum_{j=1}^{k}\,\binom{k}{j}\, 
	\p_1^{j-1}\Big(\frac{1}{P}\Big)\, \p_1^{k+1-j}\Big(\frac{1}{P}\Big)\bigg)\, \frac{\p \Delta F^{\rm WK}}{\p z_k}\nn\\
	&\+\frac{\epsilon^2}{2}\,\sum_{k,l\geq 0}\,\p_1^{k+1}\Big(\frac{1}{P}\Big)\, \p_1^{l+1}\Big(\frac{1}{P}\Big)\,
	\bigg(\frac{\p^2 \Delta F^{\rm WK}}{\p z_k\p z_l}\+\frac{\p \Delta F^{\rm WK}}{\p z_k}\frac{\p \Delta F^{\rm WK}}{\p z_l}\bigg)\nn\\
	&\+\frac{\epsilon^2}{16}\,\p_1^{k+2}\Big(\frac{1}{P^4}\Big)\,\frac{\p \Delta F^{\rm WK}}{\p z_k} \+\frac{1}{16}\,\frac{1}{P^4}=0\,.\label{eqn:loopdz}
\end{align}
Here $\p_1=\sum_{j\geq 0}z_{j+1}\frac{\p }{\p z_{j}}$, $P\=\sqrt{\lambda-z_0}$,
and 
\beq
\Delta F^{\rm WK}\=\sum_{g\geq 1}\,\epsilon^{2g-2} \, F^{\rm WK}_g(z_1,\cdots,z_{3g-2})\,.
\eeq
The following proposition was conjectured by Okuyama and Sakai in~\cite{OS},
and now we give a proof as an application of Theorem~\ref{main}.
\begin{prop}\label{prop:bgw-wkjet}
	Let $y=y(x,{\bf T})\:=Q(x,{\bf T})^2=\frac{\p^2 \F_0(x,{\bf T})}{\p T_1^2}$. For every $g\geq 1$, the genus g part of the generalized BGW free energy satisfy the identity: for $g=1$,
	\beq\label{eqn:bgwf1-wkjet}
	\F_{1}(x,{\bf T})\=\frac{1}{24}\, \log\bigg(\frac{\p y(x,{\bf T})}{\p T_1}\bigg) \,-\, \frac{\log 2}{24}\,,
	\eeq
	and for $g\geq 2$, 
	\beq\label{eqn:kw-bgwjet}
	\F_{g}(x,{\bf T})\=F^{\rm WK}_g\bigg(\frac{\p y(x,{\bf T})}{\p T_1},\cdots,\frac{\p^{3g-2} y(x,{\bf T})}{\p T_1^{3g-2}}\bigg).
	\eeq
\end{prop}
\begin{proof}
	On one hand, noticing~\eqref{flow-q} and~$u(x,{\bf T})=-\log(y(x,{\bf T}))$, we have 
	\beq\label{eqn:jetchange}
	u_x \= -\frac{y_x}{y}\,,\qquad y_x \= -\frac{y_{T_1}}{2\,y^{1/2}}\, .
	\eeq
	For $g=1$, the equality~\eqref{eqn:bgwf1-wkjet} then follows from~\eqref{eqn:bgwf1-hodgejet} and~\eqref{eqn:jetchange}.
	For $g\geq 2$, using~\eqref{eqn:bgwfg-hodgejet}, \eqref{eqn:dilatonHg}, and using~\eqref{eqn:jetchange} iteratively, we know that there exists a function $\bar F_g(z_0,z_1,\cdots,z_{3g-2})$ satisfying~\eqref{eqn:wkdilaton} such that
	\beq
	\F_{g}(x,{\bf T})\=\bar F_g\bigg(y(x,{\bf T}), \frac{\p y(x,{\bf T})}{\p T_1},\cdots,\frac{\p^{3g-2} y(x,{\bf T})}{\p T_1^{3g-2}}\bigg).
	\eeq
	 On the other hand, it follows immediately from~\eqref{flow-q} that
	\beq\label{flow-t1}
	\frac{\p y}{\p T_{2a+1}}\=\frac{y^{a}}{a!}\frac{\p y}{\p T_1}\,, \qquad a\geq 0\,.
	\eeq
	By using~\eqref{dubrovin-f0},\,\eqref{Q1}, \eqref{flow-t1},
	as well as the Virasoro constraints~\eqref{eqn:vira-gBGW-x}, one can derive the loop equation~\eqref{eqn:loopdz} using the method in~\cite{DZ-norm}.
	By the uniqueness of the solution to~\eqref{eqn:wkdilaton} and the Dubrovin-Zhang loop equation~\eqref{eqn:loopdz}, we see $\bar F_g(z,z_1,\cdots,z_{3g-2})=F^{\rm WK}_g(z_1,\cdots,z_{3g-2})$. This finishes the proof of~\eqref{eqn:kw-bgwjet}, and the proposition is proved.
\end{proof}

We note that one can also apply the theories of KdV tau-functions~\cite{DZ-norm} (cf.~also~\cite{Dickey, DYZ}) 
to prove Proposition~\ref{prop:bgw-wkjet}. For this purpose, the simple but important observation is the following:
for $g\geq2$, the power series 
\[F_g^{\rm WK} \biggl(\frac{\p y(x,T_1,{\bf 0})}{\p T_1},\cdots,\frac{\p^{3g-2} y(x,T_1,{\bf 0})}{\p T_1^{3g-2}}\biggr)\]
does not depend on~$T_1$, and so 
\beq
y(x,T_1,{\bf 0}) \+ \sum_{g\geq1} \, \hbar^{2g} \, 
\frac{\p^{2} F_g^{\rm WK} \bigl(\frac{\p y(x,T_1,{\bf 0})}{\p T_1},\cdots,\frac{\p^{3g-2} y(x,T_1,{\bf 0})}{\p T_1^{3g-2}}\bigr)}{\p T_1 \p T_1} 
 \= \frac{x^2}4 \, \frac{1}{(1-T_1)^2} \+ \frac{\hbar^2}{8} \, \frac{1}{(1-T_1)^2} \,, 
\eeq
which coincides with~\eqref{initialugBGW202218}. Here, 
$y(x,T_1,{\bf 0}) = \frac{x^2}{4 \, (1-T_1)^2}$.

We also remark that using 
\eqref{eqn:bgwf1-wkjet}--\eqref{eqn:kw-bgwjet}, \eqref{eqn:bgwfg-hodgejet}, 
\eqref{jethodge1219} and~\eqref{HodgeBGWggeq1} 
one finds that $H_g$, $g\geq1$, are 
related to $F_g^{\rm WK}$ under the substitution of the invertible map between the jet variables $z_1,z_2,\dots$ and 
$u_1,u_2,\dots$
 induced by~\eqref{eqn:jetchange}; 
an equivalent relationship with this was obtained in a joint work by Don Zagier and the first-named author of the present 
paper by a different method.

Following~\cite{IZ,Zhou1d} (cf.~also~\cite{GJV,OS,ZZ}), introduce the Itzykson--Zuber variables by
\beq\label{IZrec2022}
I_1\=1-\frac{1}{z_1}\,,\qquad
I_{\ell+1}\=\frac{1}{z_1}\,\p_1(I_{\ell})\quad (\ell\geq 1)\,.
\eeq
Clearly, $I_{k}\in \mathbb Q\big[z_2,\cdots,z_k,z_1^{-1}\big]$, $k\geq 1$. For example, 
$I_2=z_2\,z_1^{-3}$, $I_3=z_3\,z_1^{-4}-3\,z_2^2\,z_1^{-5}$.
It follows from~\eqref{euler-lagrange-Q}, \eqref{eqn:bgwf1-wkjet}, \eqref{eqn:kw-bgwjet}, \eqref{IZrec2022}
that the following identities hold:
\begin{align}
	&\F_{1}(x,{\bf T})\=\frac{1}{24}\, \log\bigg(\frac{1}{1-I_1(x,{\bf T})}\bigg) \,-\, \frac{\log 2}{24}\,,\\
	&\F_{g}(x,{\bf T})\=\frac{1}{(1-I_1(x,{\bf T}))^{2g-2}}\,\F^{\rm WK}_g 
	\bigg(0,0,\frac{I_2(x,{\bf T})}{1-I_1(x,{\bf T})},\cdots,\frac{I_{3g-2}(x,{\bf T})}{1-I_1(x,{\bf T})},{\bf 0}\bigg)\,,\label{eqn:bgw-I}
\end{align}
where $1/(1-I_1(x,{\bf T}))=\p_{T_1}(y(x,{\bf T}))\in\mathbb C[[x+2]][[{\bf T}]]$, and
\beq
I_{k}(x,{\bf T})\=\delta_{k,1}\,-\,x\,\frac{(-1)^k\,(2k-1)!!}{2^{k+1}\,Q(x,{\bf T})^{2k+1}}
\+\sum_{a\geq 0}\,\frac{Q(x,{\bf T})^{2a}}{a!}\,T_{2a+2k+1}\in\mathbb C[[x+2]][[{\bf T}]]\,,
\quad k\geq 1\,.
\eeq
In particular, by taking ${\bf T}={\bf 0}$ in~\eqref{eqn:bgw-I}, we obtain that 
\beq
\frac{x^{4g-4}}{2^{2g-2}}\,\F^{\rm WK}_g \bigg(0,0,\frac{2^1 3!!}{x^{2}},-\frac{2^25!!}{x^{4}},\cdots,(-1)^{3g-2}\frac{2^{3g-3}(6g-5)!!}{x^{6g-6}} , {\bf 0}\bigg)
\=\frac{(-1)^{g}\,2^{g-1}\,B_{2g}}{2g\,(2g-2)\,x^{2g-2}} \,.
\eeq
Using this identity and applying the derivations similar to those for~\eqref{eqn:t-kappa} (cf.~\cite{KMZ, LX, MZ}), we obtain that
\beq\label{eqn:kw-bernu}
\int_{\overline{\mathcal M}_{g,0}}\,e^{\sum_{d\geq 1}\bar s_d\,\kappa_d}
\=\frac{(-1)^{g}\,B_{2g}}{2g\,(2g-2)}\,,\qquad g\geq 2\,,
\eeq
where $\bar s_d$, $d\geq1$, are numbers determined by
\beq
\exp\bigg(-\sum_{d\geq 1}\bar s_d\,z^d\bigg)\=\sum_{d\geq 0}\,(-1)^{d}\,(2d+1)!!\,z^d\,.
\eeq
We note that the cases with $g=2,\dots,9$ of identity~\eqref{eqn:kw-bernu} were proved by Kazarian and Norbury~\cite{KN}; 
for arbitrary $g\geq 2$, this identity was also expected in~\cite{KN} to hold, 
and is now proved in this paper.

\medskip

\end{appendix}

\newpage

\noindent Di Yang

\noindent School of Mathematical Sciences, University of Science and Technology of China,

\noindent Hefei 230026, P.R. China 

\noindent diyang@ustc.edu.cn

\medskip
\medskip

\noindent Qingsheng Zhang

\noindent School of Mathematical Sciences, Peking University, 

\noindent Beijing 100871, P.R. China 

\noindent zqs@math.pku.edu.cn

\end{document}